\documentclass[11pt]{article}
\author{Richard Olu Awonusika\footnote{Corresponding author} 
	\\
	\texttt{richard.awonusika@aaua.edu.ng 
	}\\
	Department of Mathematical Sciences\\ Adekunle Ajasin University\\ P.M.B. 001, Akungba-Akoko\\ Ondo State, Nigeria.\\
}
\title{\textbf{Shifted Horadam Collocation Method for Solution of  Nonlinear Fourth-Order Boundary Value Problem in Ordinary Differential Equation}} 
\date{}

\usepackage{enumerate}
\usepackage{float}
\usepackage{hyperref}
\usepackage{mathrsfs}
\usepackage{amsthm}
\usepackage{amsmath}
\usepackage{amssymb}
\usepackage{subfigure}
\usepackage{lscape}

\usepackage[utf8]{inputenc}
\usepackage{graphicx}
\usepackage[margin=1.25 in]{geometry}

\newtheorem{example}{Example}[section]

\newtheorem{corollary}{Corollary}[section]
\newtheorem{proposition}{Proposition}[section]
\theoremstyle{definition}

\theoremstyle{remark}

\numberwithin{equation}{section}
\numberwithin{theorem}{section}

\begin{document}
	\maketitle
	
	%

	\begin{abstract}
In this paper, numerical solutions of  a class of nonlinear ordinary differential equations are obtained using a collocation method based on the shifted Horadam polynomials. The proposed problem, which is of the fourth-order, satisfies a class of two-point boundary conditions. We first discuss definitions and basic properties of the Horadam polynomials before presenting new and useful differentiation formulae for them. Novel interesting identities are deduced from the differentiation properties. The collocation method under consideration assumes that the solution of the proposed problem can be expressed as a shifted Horadam polynomial series. To determine the expansion coefficients of the series solution, one collocates at the zeros of the shifted Horadam polynomials, and the proposed boundary value problem is subsequently reduced to a set of nonlinear algebraic equations. These algebraic equations are then solved using Newton's iterative method to obtain the numerical values of the expansion coefficients of the shifted Horadam polynomial series solution. Several examples of the proposed nonlinear boundary value problem are considered to demonstrate the accuracy, efficiency,  and reliability of the proposed method. Numerical solutions and errors obtained are compared with existing solutions. Comparisons of results, which are shown in tables and graphs, clearly reveal that the shifted Horadam collocation method outperforms the existing methods. 
	\end{abstract}

	\section{Introduction}
Boundary value problems model several real-life phenomena in astrophysics (such as Lane-Emden-Fowler type problems), solid mechanics (such as cantilever beam deflection problems), fluid dynamics (such as heat transfer of pure nano fluid), mathematical physics (such as reaction-diffusion problems), hydrodynamics, astronomy, mathematical biology, economics, finance, and engineering (\cite{Agarwal1986}, \cite{Ali2019}, \cite{AwoOnuoha}, \cite{Davidson2006}, \cite{Islam2010},   \cite{Khan2003}). Higher order boundary value problems appear often in real-life applications, hence, several researchers have employed both analytical and numerical methods to study higher order boundary value problems (\cite{Akram2013}, \cite{AminRo2021}, \cite{AwoEmdenFowler},  \cite{Bishop1989}, \cite{Costabile2015}). In particular, the fourth-order boundary value problems arise in the  deflection of beam theory and analysis, plate deflection, theory, viscoelastic and inelastic flows, and deformation of beam. Many researchers have developed interests in fourth-order boundary value problems (\cite{Akinnukawe2024}, \cite{Ali2011}, \cite{Costabile2015},  \cite{Dang2010}, \cite{Dang2020}, \cite{Dang2024},  \cite{Hossain2014},  \cite{Khna2021}, \cite{Khan2024bvp}, \cite{Mehrpouya2024},   \cite{Modebei2020}, \cite{Moghadam2022},  \cite{Mohanty2000}, \cite{Noor2007}, \cite{Singh2014},  \cite{Tomar2022}, \cite{Viswan2013}, \cite{Viswana2010}). Numerical methods are capable of handling nonlinear boundary value problems more efficiently than analytical methods (\cite{Canuto1991},  \cite{Hauser2009}, \cite{Lewis2022}).

Different analytical and numerical techniques have been employed by several authors to handle nonlinear fourth-order boundary value problems. The authors in \cite{Mehrpouya2024} used Chebyshev orthogonal collocation method to obtain numerical solutions of nonlinear fourth-order boundary value problems arising in the beam theory.  Dang et al. \cite{Dang2010}, \cite{Dang2020}, and Tomar et al. \cite{Tomar2022} used iterative methods  to approximate solutions of  nonlinear fourth-order boundary value problems arising in beam analysis. In \cite{Khan2024bvp}, the authors applied linear Legendre multi-wavelets collocation method to solve fourth-order boundary value problems.  Galerkin method coupled with quintic b-splines for solvinf fourth-order boundary value problems was considered in \cite{Viswana2010}. In \cite{Ali2011}, the authors obtained numerical solutions of fourth-order boundary value problems using Haar wavelets. Variational iteration method was used in \cite{Noor2007} to solve fourth-order boundary value problems. Dang et al. \cite{Dang2024} constructed numerical methods of higher order coupled with the trapezoidal quadrature formula for fourth-order boundary value problems. In \cite{Khan2003}, the authors considered exponential solution of  boundary value problems occurring  in the plate decomposition. Modebei \cite{Modebei2020} used block hybrid method for the numerical solutions of fourth-order boundary value problems. Moghadam et al. \cite{Moghadam2022} applied the first type Lidstone collocation method to obtain solution of fourth-order boundary value problems. Decomposition method coupled with Green's function was employed in \cite{Singh2014} to obtain approximate solution of fourth-order boundary value problems. Viswanadham and Ballem \cite{Viswan2013} applied Galerkin method coupled with cubic b-splines to obtain numerical solutions of fourth-order boundary value problems. 

Further discussions on fourth-order boundary value problems can be found in \cite{Erturk2007} (differential transform method and Adomian decomposition method), \cite{Leeb2020} (via integral equations), \cite{Adak2021} (finite difference method), \cite{Hosseini2023} (Adomian decomposition method and homotopy perturbation method), \cite{Palamides2012} (solution funnel approach),   \cite{Kelesoglu2014} (Adomian decomposition method), \cite{Wazwaz2010} (modified Adomian decomposition method).  For several other methods for solving fourth and arbitrary order boundary value problems, see \cite{Bai2007},  \cite{Benaicha2016},  \cite{Li2013},  \cite{Mohanty2000}, \cite{Webb2008}. Existence and uniqueness of solutions of fourth-order boundary value problems were considered in \cite{AgarwalR2023}, \cite{Dimitrov2024}, \cite{Lv2015}, \cite{Ra02024}, \cite{Ravindra2024}, \cite{Xie2013}, \cite{Zhang2022}. For different analytical and numerical methods of solving  singular boundary value problems of Lane-Emden-Fowler type, see  \cite{Ahsan2023},  \cite{AwoEmdenFowler}, \cite{AwoOnuoha},  \cite{Gul2022},  \cite{IqbalMK2020},  \cite{KumarN2023}, \cite{PanditB2021},  \cite{PrantaSSD2023},   \cite{Qayyum2023},  \cite{Shahni2023}, \cite{Shahni2023b}, \cite{SinghR2020},  \cite{Swati2020},   \cite{VermaAK2021}, \cite{VermaAK2021-2}.


In a very recent paper, Abd-Elhameed et al. \cite{AbdElhameed2025} applied the shifted Horadam collocation method to obtain numerical solution to the nonlinear KdV equation. In this article, we apply the shifted Horadam collocation method to obtain  numerical solutions of the  two-point nonlinear fourth-order boundary value problem
\begin{equation}\label{eqlan}
	\begin{split}		
		&	u^{(4)}(x)=f(x,u(x),u'(x),u''(x),u'''(x))\qquad (0<x< 1)\\
		&u(0)=\xi_{1}, \quad u(1)=\eta_{1}, \quad u'(0)=\xi_{2}, \quad u'(1)=\eta_{2}.	
	\end{split}	
\end{equation}
Here $\xi_{i}$ and $\eta_{i}$ $(i=1,2)$ are real constants; and the function $f$ is a sufficiently smooth function. The proposed shifted Horadam collocation method expresses the solution of the boundary value problem \eqref{eqlan} as a shifted Horadam polynomial series. In determining the numerical values of the expansion coefficients of the Horadam polynomial series solution, one uses the zeros of the shifted Horadam polynomials as the collocation points. Collocating at these zeros, the proposed nonlinear boundary value problem is  reduced to a set of nonlinear algebraic equations which are  then solved using Newton's iterative method. The obtained expansion coefficients are in turn substituted into  the shifted Horadam polynomial series to obtain the required numerical solutions. Several examples of the proposed nonlinear boundary value problem are considered to illustrate the  efficiency, accuracy,  and reliability of the proposed method. Numerical solutions and errors obtained are compared with existing solutions. Comparisons of results, which are shown in tables and graphs, clearly reveal that the shifted Horadam collocation method provide better and more accurate results for all the examples considered.

The novelty and motivation of this paper are highlighted as follows:
\begin{enumerate}[(a)]
	\item Scholarly and original contributions to the subject of the proposed fourth-order boundary value problem are presented. 
	\item Six out of the seven illustrative examples of fourth-order boundary value problems considered are nonlinear.
	\item  The paper contributes, in an innovative way, to the study of the Horadam polynomials and presents some novel properties of the Horadam polynomials. New identities involving differentiation formulae for the Horadam polynomials are constructed.
	\item The paper extends the application of the Horadam polynomials to the boundary value problems in nonlinear ordinary differential equations. 
	\item Using the Horadam collocation method, the solutions obtained for all the examples considered outperform the  existing ones in    \cite{Costabile2015}, \cite{Dang2024}, \cite{Hossain2014},  \cite{Khna2021},  \cite{Khan2024bvp},   \cite{Modebei2020}, \cite{Moghadam2022},  \cite{Mohanty2000}, \cite{Noor2007}, \cite{Singh2014}, \cite{Viswan2013}.
\end{enumerate}

\section{Preliminaries}
In this section, we briefly discuss definitions and basic properties of the standard and shifted Horadam polynomials on which the proposed collocation method is based. Special cases of these polynomials are Fibonacci polynomials, Pell polynomials, Lucas polynomials, Pell-Lucas polynomials, Chebyshev polynomials of the first kind, and Chebyshev polynomials of the second kind.   

\subsection{The Horadam Polynomials}\label{Subsec Horadam}	
The $m$th-degree Horadam polynomials (\cite{Horadam1996}, \cite{Horadam1997}, \cite{Horadam1985}, \cite{Wanas2020}) $h_{m}(x)=h^{(\alpha,\beta;a,b)}_{m}(x)$ are defined by the recursive formula
\begin{align}
h_{m}(x)=&\alpha xh_{m-1}(x)+\beta h_{m-2}(x) \qquad (x\in\mathbb{R},m=2,3,\dots) \label{eq rec rel}\\
h_{0}(x)=&a, \qquad h_{1}(x)=b x, \label{eq rec relic}
\end{align}
where $\alpha,\beta;a,b$ are real constants.  The characteristic equation of the difference equation \eqref{eq rec rel} is given by
\begin{equation}
	\lambda^{2}-\alpha x\lambda-\beta=0,
\end{equation}
which has the following real roots.
\begin{equation}
\lambda_{1}=\frac{\alpha x+\sqrt{\alpha^{2}x^{2}+4\beta}}{2}, \qquad \lambda_{2}=\frac{\alpha x-\sqrt{\alpha^{2}x^{2}+4\beta}}{2}.
\end{equation}
Consequently, the Binet's form of the Horadam polynomials $h_{m}(x)$ gives
\begin{equation}
	h_{m}(x)=A_{1}\left( \frac{\alpha x+\sqrt{\alpha^{2}x^{2}+4\beta}}{2}\right)^{m-1} +A_{2}\left( \frac{\alpha x-\sqrt{\alpha^{2}x^{2}+4\beta}}{2}\right)^{m-1}, 
\end{equation}  
where
\begin{equation}
A_{1}=\frac{bx-a\lambda_{2}}{\sqrt{\alpha^{2}x^{2}+4\beta}}, \qquad  A_{1}=\frac{a\lambda_{1}-bx}{\sqrt{\alpha^{2}x^{2}+4\beta}}.
\end{equation} 
As earlier mentioned,  for particular values of the parameters $\alpha,\beta;a,b$, the Horadam polynomials $h_{m}(x)$ reduce to the following special polynomials. 
\begin{enumerate}[(i)]
	\item $a =b = \alpha =\beta =1$ \qquad (Fibonacci polynomials $F_{n}(x)$).
	\item $a = 2$ and   $b = \alpha = \beta =1$ (Lucas polynomials $L_{n}(x)$).
	\item $a = \beta = 1$ and $b = \alpha =2$  (Pell polynomials $P_{n}(x)$).
\item $a = b = \alpha = 2$ and $\beta =1$  (Pell–Lucas polynomials $Q_{n}(x)$).
\item $a=b=1$, $\alpha=2$ and $\beta=-1$     (Chebyshev polynomials  of the first kind $T_{n}(x)$).
\item $a=1,b=\alpha=2$ and $\beta=-1$ (Chebyshev polynomials of the	second kind  $U_{n}(x)$).
\end{enumerate}
The Horadam polynomials $h_{m}(x)$ admit the generating function (\cite{Wanas2020})
\begin{equation}
G(x,t)=	\frac{a+ (b-a \alpha)xt}{1-\alpha x t-\beta t^2}=\sum_{m=0}^{\infty}h_{m}(x)t^{m}.
\end{equation}
The first Horadam polynomials $h_{m}(x)$ are given as follows.
\begin{align}
h_{0}(x)=&a\nonumber\\
h_{1}(x)=&b x\nonumber\\
h_{2}(x)=&a \beta  + b \alpha  x^2\nonumber\\
h_{3}(x)=&\beta   (a \alpha +b)x+\alpha ^2 b x^3\nonumber\\
h_{4}(x)=&a \beta ^2+\alpha  \beta  (a \alpha +2 b) x^2+\alpha ^3 b x^4\nonumber\\
h_{5}(x)=&\beta ^2 (2 a \alpha +b)x +\alpha ^2 \beta   (a \alpha +3 b)x^3+\alpha ^4 b x^5\nonumber\\
h_{6}(x)=&a \beta ^3+3 \alpha  \beta ^2  (a \alpha +b)x^2+\alpha ^3 \beta  x^4 (a \alpha +4 b)+\alpha ^5 b x^6\\
h_{7}(x)=&\beta ^3  (3 a \alpha +b)x+2 \alpha ^2 \beta ^2 (2 a \alpha +3 b) x^3+\alpha ^4 \beta  (a \alpha +5 b)x^5 +\alpha ^6 b x^7\nonumber\\
h_{8}(x)=&a \beta ^4+2 \alpha  \beta ^3 (3 a \alpha +2 b) x^2+5 \alpha ^3 \beta ^2 (a \alpha +2 b)x^4 +\alpha ^5 \beta  (a \alpha +6 b)x^6 +\alpha ^7 b x^8\nonumber\\
h_{9}(x)=&\beta ^4  (4 a \alpha +b)x+10 \alpha ^2 \beta ^3(a \alpha +b) x^3 +3 \alpha ^4 \beta ^2 (2 a \alpha +5 b)x^5 +\alpha ^6 \beta  (a \alpha +7 b)x^7 +\alpha ^8 b x^9.\nonumber
\end{align}

\subsection{A Class of Shifted Horadam Polynomials}\label{Subsec ShiftedHoradam}	
The $m$th-degree shifted Horadam polynomials (\cite{AbdElhameed2025}) $\mathsf{h}_{m}(x):=h^{(\alpha,\beta;a,b)}_{m}(2x-1)$ are defined by the recursive formula
\begin{align}
\mathsf{h}_{m}(x)=&(2x-1)\alpha h_{m-1}(2x-1)+\beta h_{m-2}(2x-1) \qquad (x\in\mathbb{R},m=2,3,\dots) \label{eq rec shifrel}\\
	\mathsf{h}_{0}(x)=&a, \qquad \mathsf{h}_{1}(x)= (2x-1)b, \label{eq rec shifrelic}
\end{align}
where $\alpha,\beta;a,b$ are real constants. Now setting $a=1$, $b=\alpha=-2$, $\beta=-1$, we obtain special  $m$th-degree shifted Horadam polynomials (\cite{AbdElhameed2025}) $\mathsf{H}_{m}(x):=h^{(-2,-1;1,-2)}_{m}(2x-1)$  defined by the recursive formula
\begin{align}
	\mathsf{H}_{m}(x)=&-2(2x-1)\mathsf{H}_{m-1}(x)-\mathsf{H}_{m-2}(x) \qquad (x\in\mathbb{R},m=2,3,\dots) \label{eq rec spshifrel}\\
	\mathsf{H}_{0}(x)=&1, \qquad \mathsf{H}_{1}(x)= -2(2x-1). \label{eq rec spshifrelic}
\end{align}
The shifted Horadam polynomials $	\mathsf{H}_{m}(x)$ admit the generating function (\cite{Wanas2020})
\begin{equation}
	G(x,t)=\frac{1}{t^2+2 (2 x-1)t+1}=\sum_{m=0}^{\infty}	\mathsf{H}_{m}(x)t^{m}.
\end{equation}
The first Horadam polynomials $\mathsf{H}_{m}(x)$ are given as follows.
\begin{align}
	\mathsf{H}_{0}(x)=&1\nonumber\\
	\mathsf{H}_{1}(x)=&-4 x +2\nonumber\\
	\mathsf{H}_{2}(x)=&16 x^2-16 x+3\nonumber\\
	\mathsf{H}_{3}(x)=&-64 x^3+96 x^2-40 x+4\nonumber\\
	\mathsf{H}_{4}(x)=&256 x^4-512 x^3+336 x^2-80 x+5\nonumber\\
	\mathsf{H}_{5}(x)=&-1024 x^5+2560 x^4-2304 x^3+896 x^2-140 x+6\nonumber\\
	\mathsf{H}_{6}(x)=&4096 x^6-12288 x^5+14080 x^4-7680 x^3+2016 x^2-224 x+7\\
	\mathsf{H}_{7}(x)=&-16384 x^7+57344 x^6-79872 x^5+56320 x^4-21120 x^3+4032 x^2-336 x+8\nonumber\\
	\mathsf{H}_{8}(x)=&65536 x^8-262144 x^7+430080 x^6-372736 x^5+183040 x^4-50688 x^3+7392 x^2\nonumber\\
	&-480 x+9\nonumber\\
	\mathsf{H}_{9}(x)=&-262144 x^9+1179648 x^8-2228224 x^7+2293760 x^6-1397760 x^5+512512 x^4\nonumber\\
	&-109824 x^3+12672 x^2-660 x+10\nonumber\\
	\mathsf{H}_{10}(x)=&1048576 x^{10}-5242880 x^9+11206656 x^8-13369344 x^7+9748480 x^6-4472832 x^5\nonumber\\
	&+1281280 x^4-219648 x^3+20592 x^2-880 x+11.
\end{align}
The  Horadam polynomials $\mathsf{H}_{m}(x)$  admits the analytic finite series formulation (\cite{AbdElhameed2025})
\begin{equation}
\mathsf{H}_{m}(x)=-\label{eq hor pol}\frac{1}{2\sqrt{\pi}}\sum_{p=0}^{m}\frac{(1+2m-p)!\Gamma\left( -\frac{1}{2}-m+p\right)}{(m-p)!p!}x^{m-p}  \quad (m=0,1,2,\dots),
\end{equation}
satisfying the pointwise identities
\begin{align}\label{eq sifpo}
	\mathsf{H}_{m}(0)=m+1, \qquad \mathsf{H}_{m}(1)=(-1)^{m}(m+1).
\end{align}

We have the following higher order differentiation formulae for the shifted Horadam polynomials $\mathsf{H}_{m}(x)$.
\begin{proposition}\label{Pro diff form}
	The following differentiation formulae hold $ (m,r=1,2,\dots;m\geq r).$	
	\begin{align}
\mathsf{H}^{(r)}_{m}(x)
=&-\frac{1}{2\sqrt{\pi}}\sum_{p=0}^{m-r}\frac{(1+2m-p)!\Gamma\left( -\frac{1}{2}-m+p\right)(m-p)!}{(m-p)!p!(m-p-r)!}x^{m-p-r} \label{eq diff shifted}\\
\mathsf{H}^{(r)}_{m}(x)
=&-\frac{\Gamma \left(-m-\frac{1}{2}\right) \Gamma (2m+2)}{2 \sqrt{\pi } \Gamma (m-r+1)}\, _2F_1\left(-m-\frac{1}{2},r-m;-2 m-1;\frac{1}{x}\right) x^{m-r}. \label{eq diff shiftedhyp1}
\end{align}
\end{proposition}

\begin{proof}
Taking the $r$th-order derivative of the power function $x^{m-p}$ in the Horadam polynomials given in \eqref{eq hor pol}, we obtained the result \eqref{eq diff shifted}. On the other hand, using the hypergeometric polynomials $(k=0,1,2,\dots)$
\begin{align}
\sum_{p=0}^{k}(-1)^{p}\binom{k}{p}\frac{(b)_{p}}{(c)_{p}}z^{p}=\, _2F_1\left(-k,b;c;z\right)=\, _2F_1\left(b,-k;c;z\right),
\end{align}
with $k=m-r$, $b=-m-1/2$, $c=-2m-1$, $z=1/x$, we obtain the required result \eqref{eq diff shiftedhyp1}.
\end{proof}

As a consequence of Proposition \ref{Pro diff form}, we have the following results.
\begin{corollary}\label{Cor diff 1}	
The following differentiation property holds $($$r=1,2,\dots$$)$.
\begin{align}
	\mathsf{H}^{(r)}_{m}(0)=&-\frac{\Gamma \left(-r-\frac{1}{2}\right) (m+r+1)!}{2 \sqrt{\pi } (m-r)!},	\label{eq diff shifted1}
\end{align}
with the following special values.
\begin{align}
	\mathsf{H}'_{m}(0)=&-\frac{2(m+2) (m+1) m}{3} \quad (m=1,2,\dots)\\
	\mathsf{H}''_{m}(0)=&\frac{4 (m-1) m (m+1) (m+2) (m+3)}{15} \quad (m=2,3,\dots)\\
	\mathsf{H}'''_{m}(0)=&-\frac{ 8 (m-2) (m-1) m (m+1) (m+2) (m+3) (m+4)}{105}	\quad (m=3,4,\dots)\label{eq diff shifted11}\\
	\mathsf{H}^{(4)}_{m}(0)=&\frac{16 (m-3) (m-2) (m-1) m (m+1) (m+2) (m+3) (m+4) (m+5)}{945},m=4,5,\dots.	\label{eq diff shifted12}
\end{align}
\end{corollary}
\begin{proof}
Writing  \eqref{eq diff shifted} in the form
	\begin{align}	
\mathsf{H}^{(r)}_{m}(x)
	=&-\frac{1}{2\sqrt{\pi}}\frac{(1+m+r)!\Gamma\left( -\frac{1}{2}-r\right)r!}{r!(m-r)!}\nonumber\\
	& -\frac{1}{2\sqrt{\pi}}\sum_{p=0}^{m-r-1}\frac{(1+2m-p)!\Gamma\left( -\frac{1}{2}-m+p\right)(m-p)!}{(m-p)!p!(m-p-r)!}x^{m-p-r} 
\end{align}
and then	setting $x=0$, we obtain the result \eqref{eq diff shifted1}.  
\end{proof}
\begin{corollary}\label{Cor diff 2}
	The following differentiation formula holds $($$r=1,2,\dots$$)$.
	\begin{align}
		\mathsf{H}^{(r)}_{m}(1)=&-\frac{\Gamma \left(- \frac{2 m+1}{2}\right) \Gamma (2 m+2) \Gamma(-m)\Gamma\left( -\frac{1}{2}-r\right) }{2^{2m+3}\pi  \Gamma (m-r+1)\Gamma (-m-r-1)},
		\label{eq diff shifted-1}
	\end{align}
	with the following special values.
	\begin{align}
		\mathsf{H}'_{m}(1)=&-\frac{2 \Gamma (-2 m-1) \Gamma (2 m+2)}{3 \Gamma (-m-2) \Gamma (m)} \quad (m=1,2,\dots)\\
		\mathsf{H}''_{m}(1)=&\frac{4 \Gamma (-2 m-1) \Gamma (2 m+2)}{15 \Gamma (-m-3) \Gamma (m-1)} \quad (m=2,3,\dots)\label{eq diff shifted-11}\\
		\mathsf{H}'''_{m}(1)=&-\frac{8 \Gamma (-2 m-1) \Gamma (2 m+2)}{105 \Gamma (-m-4) \Gamma (m-2)} \quad (m=3,4,\dots) \\
		\mathsf{H}^{(4)}_{m}(1)=&\frac{16 \Gamma (-2 m-1) \Gamma (2 m+2)}{945 \Gamma (-m-5) \Gamma (m-3)} \quad (m=4,5,\dots).\label{eq diff shifted-12}
	\end{align}
	Note that $\Gamma(-m-r)=-(m+r+1)\Gamma(-(m+r+1))$ $(m=r-1,r,r+1,\dots;r=2,3,\dots)$.
\end{corollary}
\begin{proof}
 Setting  $x=1$ in \eqref{eq diff shiftedhyp1} and then using the well-known formula 
	\begin{align}
		\, _2F_1\left(a,b;c;1\right)=\frac{\Gamma(c)\Gamma(c-a-b)}{\Gamma(c-a)\Gamma(c-b)},
	\end{align}
	with $a=-m-1/2$, $b=-m+r$,  $c=-2m-1$, we obtain \eqref{eq diff shifted-1} as required. 
\end{proof}





\begin{proposition}[\cite{AbdElhameed2025}]\label{Pro diff forMAbd}
	The following differentiation formula holds $ (m,r=1,2,\dots;m\geq r).$	
	\begin{align}
	\mathsf{H}^{(r)}_{m}(x)	&=\sum_{p=0}^{m-r}\Theta^{m}_{p,r}	\mathsf{H}_{p}(x), 
	\end{align}
	where
		\begin{align}\label{eq coetheta1}
		\Theta^{m}_{p,r}=	\frac{(-1)^{m-p}2^{2r}(p+1)\left(\frac{m+p-r }{2}+2\right)_{r-1}(r)_{\frac{m-p-r }{2}}}{\left(\frac{m-p-r }{2}\right)!}\theta^{m}_{p,r},
	\end{align}
	with
	\begin{align}\label{eq coetheta2}
		\theta^{m}_{p,r}=\begin{cases}
			1, & m-p-r \,\, \mbox{even},\\
			0, & \mbox{otherwise}.
		\end{cases}
	\end{align}
Here $(x)_{k}=\Gamma(x+k)/\Gamma(x)$ $(k=0,1,2,\dots.)$
\end{proposition}

From the differentiation formulae \eqref{eq diff shifted1} and \eqref{eq diff shifted-1}, Proposition \ref{Pro diff forMAbd} and equation \eqref{eq sifpo}, we obtain the following identities given as  corollaries.

%

\begin{corollary}
The following identity holds $ (m,r=1,2,\dots).$	
\begin{align}
&\sum_{p=0}^{m-r}\Theta^{m}_{p,r}	(p+1)=-\frac{\Gamma \left(-r-\frac{1}{2}\right) (m+r+1)!}{2 \sqrt{\pi } (m-r)!},
\end{align}
with the following special values.
\begin{align}
&\sum_{p=0}^{m-1}(-1)^{m-p}	(p+1)^{2}\theta^{m}_{p,1}=-\frac{(m+2) (m+1) m}{6} \quad (m=1,2,\dots)\\
&\sum_{p=0}^{m-2}(-1)^{m-p}	(p+1)^{2}(m-p)(m+p+2)\theta^{m}_{p,2}\nonumber\\
=&\frac{ (m-1) m (m+1) (m+2) (m+3)}{15} \quad (m=2,3,\dots)\\
&\sum_{p=0}^{m-3}(-1)^{m-p}	(p+1)^{2}(p-m-1)(p-m+1)(m+p+1)(m+p+3)\theta^{m}_{p,3}\nonumber\\
=&-\frac{4 (m-2) (m-1) m (m+1) (m+2) (m+3) (m+4)}{105}	\quad (m=3,4,\dots).
\end{align}
Here $	\Theta^{m}_{p,r}$ is as given in \eqref{eq coetheta1}-\eqref{eq coetheta2}.
\end{corollary}

\begin{corollary}
	The following identity holds  $ (m,r=1,2,\dots).$	
	\begin{align}
		&\sum_{p=0}^{m-r}\Theta^{m}_{p,r}	(-1)^{p}(p+1)=-\frac{\Gamma \left(- \frac{2 m+1}{2}\right) \Gamma (2 m+2) \Gamma(-m)\Gamma\left( -\frac{1}{2}-r\right) }{2^{2m+3}\pi  \Gamma (m-r+1)\Gamma (-m-r-1)},
		\label{eq diff shifted-d1}
	\end{align}
	with the following special values.
	\begin{align}
		&\sum_{p=0}^{m-1}(-1)^{m}	(p+1)^{2}\theta^{m}_{p,1}=-\frac{\Gamma (-2 m-1) \Gamma (2 m+2)}{6 \Gamma (-m-2) \Gamma (m)} \quad (m=1,2,\dots)\\
		&\sum_{p=0}^{m-2}(-1)^{m}	(p+1)^{2}(m-p)(m+p+2)\theta^{m}_{p,2}\nonumber\\
		=&\frac{ \Gamma (-2 m-1) \Gamma (2 m+2)}{15 \Gamma (-m-3) \Gamma (m-1)} \quad (m=2,3,\dots)\label{eq diff shifted-11de}\\
		&\sum_{p=0}^{m-3}(-1)^{m}	(p+1)^{2}(p-m-1)(p-m+1)(m+p+1)(m+p+3)\theta^{m}_{p,3}\nonumber\\
		=&-\frac{4 \Gamma (-2 m-1) \Gamma (2 m+2)}{105 \Gamma (-m-4) \Gamma (m-2)} \quad (m=3,4,\dots) 
	\end{align}
	Here $	\Theta^{m}_{p,r}$ is as given in \eqref{eq coetheta1}-\eqref{eq coetheta2}.
\end{corollary}

Let 
\begin{equation}
	\mathcal{H}_{N}=\mbox{span}\left\lbrace\mathsf{H}_{m}(x):0\leq m\leq N \right\rbrace. 
\end{equation}
Any function $u=u(x)\in \mathcal{H}_{N}$  may
be expressed in terms of the shifted Horadam polynomials as
\begin{equation}\label{eq soly}
	u_{N}(x)=\sum_{m=0}^{N}c_{m}\mathsf{H}_{m}(x) \qquad (N\in\mathbb{N}),	
\end{equation}
where the coefficients $c_{m}$ $(0\leq m\leq N)$ are to be determined.


\section{Horadam  Collocation Method of Solution} \label{Sec SHCM}

This section presents the  proposed collocation method based on the shifted Horadam polynomials for solving the fourth-order boundary value problem \eqref{eqlan}.  In this method, one assumes that  the solution $u(x)$ can be expressed as a shifted Horadam polynomial series formulation \eqref{eq soly}. Collocating at the $N-3$ zeros of the shifted Horadam polynomials $\mathsf{H}_{N-3}(x)$, and upon applying the four boundary conditions in \eqref{eqlan}, we obtain a set of $N+1$ nonlinear algebraic equations. These algebraic equations are then solved for  the $N+1$ unknown expansion coefficients  using Newton's iterative method.  


To this end,  the fourth-order nonlinear boundary value problem under consideration is given by
\begin{equation}\label{eqf lepic}
	\begin{split}		
		&	u^{(4)}(x)=f(x,u(x),u'(x),u''(x),u'''(x))\qquad (0<x< 1)\\
		&u(0)=\xi_{1}, \quad u(1)=\eta_{1} \quad u'(0)=\xi_{2}, \quad u'(1)=\eta_{2}.	
	\end{split}
\end{equation}
 We apply the Horadam collocation method by assuming that the solution of the boundary value problem \eqref{eqf lepic}	 can be expressed as the shifted Horadam polynomial series formulation
\begin{align}\label{eq soluv}
	\begin{split}
		u_{N}(x)=&\sum_{m=0}^{N}c_{m}\mathsf{H}_{m}(x) \qquad (N\in\mathbb{N}),
	\end{split}
\end{align}
where the shifted Horadam expansion coefficients $c_{m}$ $(m=0,1,2,\dots,N)$ are to be determined. Substituting the series solution \eqref{eq soluv} into the differential equation in \eqref{eqf lepic}, we have
\begin{equation}\label{eqf lepic1}
	\begin{split}		
		&\sum_{m=0}^{N}c_{m}\mathsf{H}^{(4)}_{m}(x) =f\left( x,\sum_{m=0}^{N}c_{m}\mathsf{H}_{m}(x) ,\sum_{m=0}^{N}c_{m}\mathsf{H}'_{m}(x) ,\sum_{m=0}^{N}c_{m}\mathsf{H}''_{m}(x) ,\sum_{m=0}^{N}c_{m}\mathsf{H}'''_{m}(x) \right),
	\end{split}
\end{equation}
satisfying the boundary conditions
\begin{align}
	&\sum_{m=0}^{N}c_{m}\mathsf{H}_{m}(0)=\xi_{1}, \quad \sum_{m=0}^{N}c_{m}\mathsf{H}_{m}(1)=\eta_{1}\\
	&\sum_{m=0}^{N}c_{m}\mathsf{H}'_{m}(0)= \xi_{2}, \qquad 	\sum_{m=0}^{N}c_{m}\mathsf{H}'_{m}(1) =\eta_{2}.	
\end{align}
Using the properties \eqref{eq sifpo} and the   differentiation formulae in Proposition \ref{Pro diff form}, Corollaries \ref{Cor diff 1} and \ref{Cor diff 2}, we obtain
\begin{align}\label{eqf lepic2-1}
		&-\frac{1}{2\sqrt{\pi}}\sum_{m=4}^{N}c_{m}\sum_{p=0}^{m-4}\frac{(1+2m-p)!\Gamma\left( -\frac{1}{2}-m+p\right)(m-p)!}{(m-p)!p!(m-p-4)!}x^{m-p-4}\nonumber\\
		& =f\left( x,-\frac{1}{2\sqrt{\pi}}\sum_{m=0}^{N}c_{m}\sum_{p=0}^{m}\frac{(1+2m-p)!\Gamma\left( -\frac{1}{2}-m+p\right)}{(m-p)!p!}x^{m-p} ,\right.\nonumber\\
		&\left.-\frac{1}{2\sqrt{\pi}}\sum_{m=1}^{N}c_{m}\sum_{p=0}^{m-1}\frac{(1+2m-p)!\Gamma\left( -\frac{1}{2}-m+p\right)(m-p)!}{(m-p)!p!(m-p-1)!}x^{m-p-1},\right.\nonumber\\
		&\left.-\frac{1}{2\sqrt{\pi}}\sum_{m=2}^{N}c_{m}\sum_{p=0}^{m-2}\frac{(1+2m-p)!\Gamma\left( -\frac{1}{2}-m+p\right)(m-p)!}{(m-p)!p!(m-p-2)!}x^{m-p-2},\right.\nonumber\\
		&\left.-\frac{1}{2\sqrt{\pi}}\sum_{m=3}^{N}c_{m}\sum_{p=0}^{m-3}\frac{(1+2m-p)!\Gamma\left( -\frac{1}{2}-m+p\right)(m-p)!}{(m-p)!p!(m-p-3)!}x^{m-p-3} \right),
\end{align}
satisfying the boundary conditions
\begin{align}\label{eq bccolo}
	&\sum_{m=0}^{m}(m+1)c_{m}=\xi_{1}, \quad 	\sum_{m=0}^{N}(-1)^{m}(m+1)c_{m}=\eta_{1}\nonumber\\
	&-\frac{2}{3}\sum_{m=1}^{N}(m+2) (m+1)mc_{m}= \xi_{2}, \qquad 	-\frac{2}{3}\sum_{m=1}^{N}\frac{ \Gamma (-2 m-1) \Gamma (2 m+2)}{\Gamma (-m-2) \Gamma (m)}c_{m} =\eta_{2}.	
\end{align}
%
Equation  \eqref{eqf lepic2-1} is satisfied exactly at the collocation points $\mathsf{x}_{q,N}$ $(0\leq q\leq N-4)$ which are zeros of the  shifted Horadam polynomials $\mathsf{H}_{N-3}(x)$. Thus we have a collocation scheme $(0\leq q\leq N-4)$
\begin{align}\label{eqf lepic3}
		&-\frac{1}{2\sqrt{\pi}}\sum_{m=4}^{N}c_{m}\sum_{p=0}^{m-4}\frac{(1+2m-p)!\Gamma\left( -\frac{1}{2}-m+p\right)(m-p)!}{(m-p)!p!(m-p-4)!}\left( \mathsf{x}_{q,N}\right)^{m-p-4}\nonumber\\
& =f\left( \mathsf{x}_{q,N},-\frac{1}{2\sqrt{\pi}}\sum_{m=0}^{N}c_{m}\sum_{p=0}^{m}\frac{(1+2m-p)!\Gamma\left( -\frac{1}{2}-m+p\right)}{(m-p)!p!}\left( \mathsf{x}_{q,N}\right)^{m-p} ,\right.\nonumber\\
&\left.-\frac{1}{2\sqrt{\pi}}\sum_{m=1}^{N}c_{m}\sum_{p=0}^{m-1}\frac{(1+2m-p)!\Gamma\left( -\frac{1}{2}-m+p\right)(m-p)!}{(m-p)!p!(m-p-1)!}\left( \mathsf{x}_{q,N}\right)^{m-p-1},\right.\nonumber\\
&\left.-\frac{1}{2\sqrt{\pi}}\sum_{m=2}^{N}c_{m}\sum_{p=0}^{m-2}\frac{(1+2m-p)!\Gamma\left( -\frac{1}{2}-m+p\right)(m-p)!}{(m-p)!p!(m-p-2)!}\left( \mathsf{x}_{q,N}\right)^{m-p-2},\right.\nonumber\\
&\left.-\frac{1}{2\sqrt{\pi}}\sum_{m=3}^{N}c_{m}\sum_{p=0}^{m-3}\frac{(1+2m-p)!\Gamma\left( -\frac{1}{2}-m+p\right)(m-p)!}{(m-p)!p!(m-p-3)!}\left( \mathsf{x}_{q,N}\right)^{m-p-3} \right).
\end{align}
Hence, the collocation scheme \eqref{eqf lepic3} together with the four initial-boundary collocation schemes \eqref{eq bccolo} reduces  the given boundary value problem \eqref{eqf lepic} to a set of $(N+1)$ nonlinear algebraic equations for $(N+1)$ unknown  shifted Horadam expansion coefficients $c_{m}$ $(m=0,1,2,\dots,N)$.  These algebraic equations are then solved using Newton's iterative method.

\paragraph*{Convergence and error analysis.} The convergence and error analysis of the present collocation method is discussed extensively in \cite{AbdElhameed2025}.

\section{Illustrative Examples} 
In this section, we  compute numerical solutions of several examples of the fourth-order boundary value problem \eqref{eqlan} using the Horadam polynomials collocation algorithm presented in Section \ref{Sec SHCM}. All the examples considered are nonlinear except Example \ref{Example 4.1}. We use Wolfram Mathematica 12.0 software for our  computations.

To this end, 
we define the absolute error, AE, by 
\begin{align}\label{eqekey 5}
	\begin{split}
		\left|u_{{\rm ex.}}(x) -u_{N}(x)\right|
	\end{split}
\end{align}
and the maximum absolute error, MAE, by
\begin{equation}
	\max\limits_{x\in(0,1)}\left|u_{{\rm ex.}}(x)-u_{N}(x) \right|. 
\end{equation}
Here $u_{{\rm ex.}}(x)$  represents the exact solution  and 
$u_{N}(x)$ denotes the approximate solution.  Our results are compared with results in a recent article of Dang et al. \cite{Dang2024} and other notable existing results. In all the examples considered, we set $N=16$. 

%
%

\begin{example}\label{Example 4.1}
	Consider  the nonlinear boundary value problem  (\cite{Costabile2015},  \cite{Dang2024}, \cite{Erturk2007}, \cite{Hosseini2023}, \cite{Moghadam2022}, \cite{Momani2006},  \cite{Noor2007}, \cite{Singh2014}, \cite{Viswan2013})
	\begin{equation}\label{eqflan 34th-11}
		\begin{split}		
			&u^{(4)}(x)-u''(x)- u(x) =(x-3)e^{x} \qquad (0<x< 1)
			\\
			&u(0)=0,\quad u(1)=0, \quad u'(0)=0, \quad u'(1)=-e.
		\end{split}
	\end{equation}
	The exact solution  is $u(x)=(1-x)e^{x}$.

Specialising the shifted Horadam collocation algorithm in Section \ref{Sec SHCM} to this example, we obtain the expansion coefficients of the shifted Horadam polynomial series as follows:	
\begin{align}
c_0=&0.745183, \quad c_1=0.227862, \quad c_2= -0.0801293, \quad c_3= 0.0109772\nonumber\\
c_4=&-0.000954681,\quad c_5=0.0000611606, \quad c_6= -3.10\times 10^{-6}, \quad c_7= 1.30\times 10^{-7}\nonumber\\
 c_8=& 4.71\times 10^{-9}, \quad c_9= 1.48\times 10^{-10}, \quad c_{10}= -4.14\times 10^{-12}, \quad c_{11}=1.03\times 10^{-13}\nonumber\\
 c_{12}= &-2.37\times 10^{-15}, \quad c_{13}= 4.95\times 10^{-17}, \quad c_{14}= -9.55\times 10^{-19}, \quad c_{15}= 1.71\times 10^{-20}\nonumber\\
 c_{16}=&-2.84\times 10^{-22}.	
\end{align}	
Hence, we obtain the approximate Horadam polynomial series solution
\begin{align}
u(x)=&1.00000-1.60\times 10^{-18} x-0.50000 x^2-0.333333 x^3-0.12500 x^4-0.0333333 x^5\nonumber\\
&-0.00694444 x^6-0.00119048 x^7-0.000173611 x^8-0.0000220459 x^9\nonumber\\
&-2.48\times 10^{-6} x^{10}-2.50\times 10^{-7} x^{11}-2.29\times 10^{-8} x^{12}-1.91\times 10^{-9} x^{13}\nonumber\\
&-1.54\times 10^{-10} x^{14}-8.60\times 10^{-12} x^{15}-1.22\times 10^{-12} x^{16}.
\end{align}
\end{example}

Table \ref{table Example 4.1solution} shows the comparison of the exact solution, solution obtained from the present method, and the variational iteration method \cite{Noor2007}. Comparison of the absolute error and maximum absolute error of the present method and other existing methods is presented in Table \ref{table Example 4.1error}. It is observed from Table \ref{table Example 4.1error} and the comparative analysis in \cite{Dang2024} that, for the same value of $N$, the present method outperforms the methods presented in \cite{Dang2024} and \cite{Costabile2015}. The graphs of the comparison of the exact and SHC solutions are plotted in Figure \ref{fig:Example4.1a} and the graph of the absolute errors is shown in Figure \ref{fig:Example4.1b}.

\begin{table}[H]
	\caption {Comparison of exact solution with the present method and VIM (\cite{Noor2007}) for Example \ref{Example 4.1}.}
	\centering 
	\begin{tabular}{|c| c| c |c|} 
		\hline
		$x$ &Exact Solution&Present Method &Variational IM \cite{Noor2007}
		\\ [0.5ex] 
		\hline 
		$0.1$ & $0.9946538262680830$&$ 0.9946538262680832$& $0.9946538264$ 
		\\
		\hline
		$0.2$ & $0.9771222065281360$&$ 0.9771222065281364$& $0.9771222072$ 
		\\
		\hline
		$0.3$ & $0.9449011653032021$&$ 0.9449011653032028$& $0.9449011333$ 
		\\
		\hline
		$0.4$ & $0.8950948185847621$&$ 0.8950948185847627$& $0.8950948205$ 
		\\
		\hline
		$0.5$ & $0.8243606353500641$&$ 0.8243606353500647$& $0.8243606378$  
		\\
		\hline
		$0.6$ &$0.7288475201562036$&$ 0.7288475201562042$&$0.7288475228$  
		\\
		\hline
		$0.7$ & $0.6041258122411430$&$ 0.6041258122411437$& $0.6041258147$ 
		\\
		\hline
		$0.8$ & $0.4451081856984935$&$ 0.4451081856984941$& $0.4451081875$ 
		\\
		\hline
		$0.9$ & $0.2459603111156949$&$ 0.2459603111156956$& $0.2489603118$ 
		\\
		\hline
	\end{tabular}
	\label{table Example 4.1solution} 
\end{table}

\begin{landscape}
\begin{table}[H]
	\caption {Comparison of absolute errors of the present method  and the methods presented in \cite{Costabile2015},  \cite{Dang2024}, 
		\cite{Hosseini2023}, \cite{Moghadam2022}, 
\cite{Noor2007} 
 for Example \ref{Example 4.1}.}
	\centering 
	\begin{tabular}{|c| c| c |c|c|c|c|} 
		\hline
		$x$ & SHCM $(N=16)$&Method  \cite{Dang2024} &Method  \cite{Hosseini2023}&Method  \cite{Moghadam2022}&Method  \cite{Costabile2015}& Method \cite{Noor2007} $(N=15)$ 
		\\ [0.5ex] 
		\hline 
		$0.1$ & $3.53\times10^{-16}$ & -&$7.11\times10^{-8}$& $7.69\times10^{-10}$ & $6.33\times10^{-9}$  & $2.00\times10^{-10}$ 
		\\
		\hline
		$0.2$ & $3.75\times10^{-16}$ &-&$9.62\times10^{-8}$& $1.46\times10^{-9}$ & $7.50\times10^{-9}$  & $7.01\times10^{-10}$ 
		\\
		\hline
		$0.3$ & $5.84\times10^{-16}$ &-& $2.10\times10^{-7}$&$2.02\times10^{-9}$ & $3.12\times10^{-10}$ &  $1.35\times10^{-9}$ 
		\\
		\hline
		$0.4$ & $6.08\times10^{-16}$ &-& $1.16\times10^{-7}$&$3.39\times10^{-9}$ & $3.91\times10^{-9}$ &  $2.00\times10^{-9} $ 
		\\
		\hline
		$0.5$ & $6.04\times10^{-16}$ &-& $9.30\times10^{-8}$&$2.54\times10^{-9}$  &$9.89\times10^{-10}$  & $2.51\times10^{-9}$  
		\\
		\hline
		$0.6$ &$6.96\times10^{-16}$  &-&$1.21\times10^{-7}$&$2.44\times10^{-9}$  & $1.75\times10^{-9}$ & $2.72\times10^{-9}$ 
		\\
		\hline
		$0.7$ & $6.77\times10^{-16}$ &-& $3.19\times10^{-7}$&$2.11\times10^{-9}$ & $2.91\times10^{-9}$ & $2.21\times10^{-9}$  
		\\
		\hline
		$0.8$ & $5.30\times10^{-16}$ &-& $2.33\times10^{-7}$&$1.55\times10^{-9}$ & $1.01\times10^{-8}$  & $1.80\times10^{-9}$ 
		\\
		\hline
		$0.9$ & $5.99\times10^{-16}$ &-&$6.99\times10^{-8}$& $8.27\times10^{-10}$ & $7.80\times10^{-9}$  & $7.25\times10^{-10}$ 
		\\
		\hline
		\hline
		MAE &$6.96\times10^{-16}$ & $3.26\times 10^{-14}(N=16)$ &$3.19\times10^{-7}$&$3.39\times 10^{-9}$ & $1.01\times 10^{-8}$   & $2.72\times 10^{-9}$  
		\\
	&&$1.29\times 10^{-16}(N=32)$&&  
	&&
	\\ [0.5ex] 	
		\hline	
	\end{tabular}
	\label{table Example 4.1error} 
\end{table}
\end{landscape}
\begin{figure}[h!]
	\begin{center}
		\subfigure[Exact and numerical (SHCM) solutions.]{%
	\label{fig:Example4.1a}
	\includegraphics[width=0.475\textwidth]{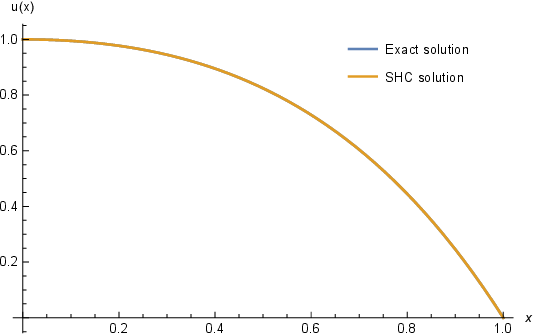}
}
\subfigure[Absolute errors of SHCM.]{%
	\label{fig:Example4.1b}
	\includegraphics[width=0.475\textwidth]{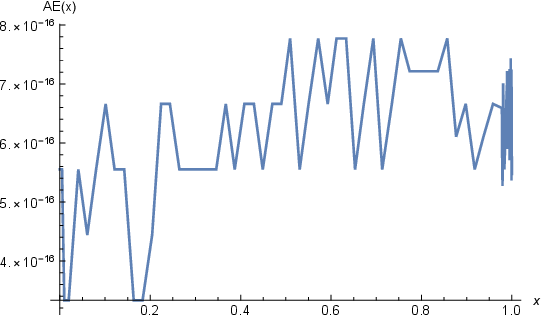}
}
	\end{center}
	\caption{Comparison of exact and numerical solutions for Example \ref{Example 4.1}.
	}%
	\label{fig:Example 4.1}
\end{figure}

\begin{example}\label{Example 4.2}
	Consider  the nonlinear boundary value problem  (\cite{Moghadam2022}, \cite{Singh2014})
	\begin{equation}\label{eqflan 3th-11}
		\begin{split}		
			&u^{(4)}(x)=e^{-x} u^{2}(x) \qquad (0<x< 1)
			\\
			&u(0)=1,\quad u(1)=e, \quad u'(0)=1, \quad u'(1)=e.
		\end{split}
	\end{equation}
	The exact solution  is $u(x)=e^{x}$.

Narrowing down the shifted Horadam collocation algorithm in Section \ref{Sec SHCM} to the boundary value problem \eqref{eqflan 3th-11}, we get the expansion coefficients 	
\begin{align}
	c_0=&1.70078, \quad c_1=-0.420835, \quad c_2=0.0523326, \quad c_3= -0.00434749\nonumber\\
	c_4=&0.000271154,\quad c_5=-0.0000135376, \quad c_6= 5.63438\times 10^{-7}, \quad c_7= -2.01054\times 10^{-8}\nonumber\\
	c_8=& 6.27856\times 10^{-10}, \quad c_9=-1.74306\times 10^{-11}, \quad c_{10}=4.35558\times 10^{-13}\nonumber\\ c_{11}=&-9.89508\times 10^{-15}, \quad 
	c_{12}= 2.06076\times 10^{-16}, \quad c_{13}=-3.96172\times 10^{-18}\nonumber\\ c_{14}=&7.07533\times 10^{-20}, \quad c_{15}=-1.20753\times 10^{-21}, \quad
	c_{16}=1.97687\times 10^{-23}.	
\end{align}	
Thus one has the following approximate solution
\begin{align}
	u(x)=&1.00000+1.00000 x+0.50000 x^2+0.166667 x^3+0.0416667 x^4+0.00833333 x^5\nonumber\\
	&+0.00138889 x^6+0.000198413 x^7+0.0000248016 x^8+2.75573\times 10^{-6} x^9\nonumber\\
	&+2.75572\times 10^{-7} x^{10}+2.50527\times 10^{-8} x^{11}+2.08741\times 10^{-9} x^{12}\nonumber\\
	&+1.60432\times 10^{-10} x^{13}+1.17359\times 10^{-11} x^{14}+6.17327\times 10^{-13} x^{15}\nonumber\\
	&+8.49057\times 10^{-14} x^{16}.
\end{align}
		
\end{example}

Table \ref{table Example 4.2} illustrates the comparison between the exact solution, SHC solution, as well as absolute errors obtained from the present SHCM and other existing methods. It is clearly  observed that the present method performs better than the method in a recent paper \cite{Moghadam2022}, and other existing methods in \cite{Singh2014}. In Figure \ref{fig:Example4.2a}, we plot the graphs of the exact and SHC solutions, while Figure \ref{fig:Example4.2b} presents the graph of the absolute errors of the present method. 

\begin{table}[H]
	\caption {Comparison of absolute errors of the present method  and those in \cite{Moghadam2022}, \cite{Singh2014} for Example \ref{Example 4.2}.}
	\centering 
	\begin{tabular}{|c| c| c |c|c|c|} 
		\hline
	$x$ &Exact Solution&SHC Method & SHC Method &Method \cite{Moghadam2022} &Method \cite{Singh2014} 
		\\ [0.5ex] 
		\hline 
		$0.1$ &$1.1051709180756477$&$1.1051709180756477$& $8.38491\times10^{-17}$ & -& - 
		\\
				\hline 
		$0.2$ &$1.2214027581601699$&$1.2214027581601699$& $1.37748\times10^{-17}$ & $2.6534\times10^{-10}$ & $3.5093\times10^{-8}$  
		\\
				\hline 
	$0.3$ &$1.3498588075760032$&$1.3498588075760030$& $1.15942\times10^{-16}$ & - & -  
	\\
		\hline
		$0.4$ &$1.4918246976412703$&$ 1.4918246976412703$& $2.77127\times10^{-17}$ & $4.3282\times10^{-10}$ & $8.2634\times10^{-8}$  
		\\
		\hline
		$0.5$ &$1.6487212707001282$&$ 1.6487212707001282$& $9.28217\times10^{-17}$ & - & -  
\\
\hline
		$0.6$ & $1.8221188003905089$&$ 1.8221188003905090$&$1.76347\times10^{-17}$  &$4.4097\times10^{-10}$  & $8.2376\times\times10^{-8}$   
		\\
	\hline
	$0.7$ & $2.0137527074704766$&$ 2.0137527074704760$&$2.94333\times10^{-16}$  &- & -  
	\\
		\hline
		$0.8$ &$2.2255409284924680$&$ 2.2255409284924680$&$2.22131\times10^{-17}$  & $2.7962\times10^{-10}$ & $3. 4819\times\times10^{-8}$   
		\\
			\hline
	$0.9$ &$2.4596031111569500$&$ 2.4596031111569494$&$1.84535\times10^{-16}$  & - & -  
	\\
	\hline
\hline
MAE&&&$2.94333\times10^{-16}$&$4.4097\times10^{-10}$&$8.2634\times10^{-8}$ 
\\
		\hline
	\end{tabular}
	\label{table Example 4.2} 
\end{table}

\begin{figure}[h!]
	\begin{center}
		\subfigure[Exact and numerical (SHCM) solutions.]{%
			\label{fig:Example4.2a}
			\includegraphics[width=0.475\textwidth]{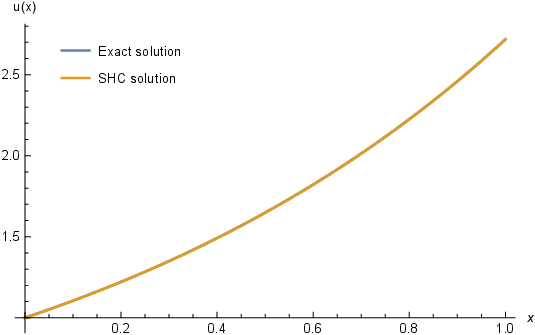}
		}
		\subfigure[Absolute errors of SHCM.]{%
			\label{fig:Example4.2b}
			\includegraphics[width=0.475\textwidth]{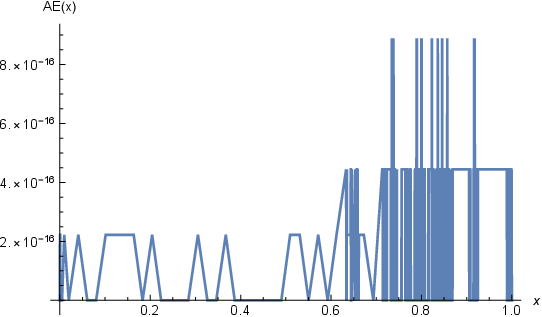}
		}
	\end{center}
	\caption{Comparison of exact and numerical solutions for Example \ref{Example 4.2}.
	}%
	\label{fig:Example 4.2}
\end{figure}

\begin{example}\label{Example 4.3}
Consider  the nonlinear boundary value problem  (\cite{Dang2024},  \cite{Hossain2014}, \cite{Hosseini2023}, \cite{Modebei2020}, \cite{Noor2007})
	\begin{equation}\label{eqflan 34th-1}
		\begin{split}		
			&u^{(4)}(x)+ (u''(x))^{2} =\sin x+ \sin^{2}x \qquad (0<x< 1)
			\\
			&u(0)=0,\quad u(1)=\sin 1, \quad u'(0)=1, \quad u'(1)=\cos 1.
		\end{split}
	\end{equation}
	The exact solution  is $u(x)=\sin x$. The shifted Horadam collocation algorithm in Section \ref{Sec SHCM} gives the following  expansion coefficients 	for this example.
\begin{align}
	c_0=&0.464599, \quad c_1=-0.21486, \quad c_2=-0.0147494, \quad c_3=0.00225695\nonumber\\
	c_4=&0.0000772223,\quad c_5=-7.07827\times 10^{-6}, \quad c_6=-1.613\times 10^{-7}, \quad c_7= 1.05541\times 10^{-8}\nonumber\\
	c_8=& 1.80304\times 10^{-10}, \quad c_9=-9.17312\times 10^{-12}, \quad c_{10}=-1.25342\times 10^{-13}\nonumber\\ c_{11}=&5.21656\times 10^{-15}, \quad 
	c_{12}=5.93916\times 10^{-17}, \quad c_{13}=-2.09131\times 10^{-18}\nonumber\\ c_{14}=&-2.04009\times 10^{-20}, \quad c_{15}=6.22061\times 10^{-22}, \quad
	c_{16}=4.73335\times 10^{-24}.	
\end{align}	
Hence, the approximate series solution yields
\begin{align}
	u(x)=&1.73179\times 10^{-17}+1.00000 x+2.00010\times 10^{-17} x^2-0.166667 x^3+1.44927\times 10^{-18} x^4\nonumber\\
	&+0.00833333 x^5+6.25754\times 10^{-17} x^6-0.000198413 x^7+3.05910\times 10^{-15} x^8\nonumber\\
	&+2.75573\times 10^{-6} x^9+2.73963\times 10^{-14} x^{10}-2.50521\times 10^{-8} x^{11}\nonumber\\
	&+1.20147\times 10^{-13} x^{12}+1.60441\times 10^{-10} x^{13}+1.24003\times 10^{-13} x^{14}\nonumber\\
	&-8.30569\times 10^{-13} x^{15}+2.03295\times 10^{-14} x^{16}.
\end{align}

\end{example}

Table \ref{table Example 4.3solution} shows the comparison of the exact solution, solution obtained from the present method, and the variational iteration method \cite{Noor2007}. Comparison of the absolute error and maximum absolute error of the present method and other existing methods is presented in Table \ref{table Example 4.3error}. It is seen from Table \ref{table Example 4.3error} and the comparative analysis in \cite{Dang2024} that, for the same value of $N$, the present method provides better approximations compared to  the methods presented in \cite{Dang2024}, \cite{Costabile2015}, and \cite{Noor2007}. The graphs of the comparison of the exact and SHC solutions are plotted in Figure \ref{fig:Example4.3a} and the graph of the absolute errors is shown in Figure \ref{fig:Example4.3b}.

\begin{table}[H]
	\caption {Comparison of exact solution with the present method and VIM (\cite{Noor2007}) for Example \ref{Example 4.3}.}
	\centering 
	\begin{tabular}{|c| c| c |c|} 
		\hline
		$x$ &Exact Solution&Present Method &Variational IM \cite{Noor2007}
		\\ [0.5ex] 
		\hline 
		$0.1$ & $0.09983341664682815$&$ 0.09983341664682817$& $0.0998334945$ 
		\\
		\hline
		$0.2$ & $0.19866933079506122$&$ 0.19866933079506124$& $0.1986696031$ 
		\\
		\hline
		$0.3$ & $0.29552020666133955$&$ 0.29552020666133950$& $0.2955207315$ 
		\\
		\hline
		$0.4$ & $0.38941834230865050$&$ 0.38941834230865047$& $0.3894191196$ 
		\\
		\hline
		$0.5$ & $0.47942553860420300$&$ 0.47942553860420295$& $0.4794265100$  
		\\
		\hline
		$0.6$ &$0.56464247339503540$&$ 0.56464247339503510$&$0.5646435236$  
		\\
		\hline
		$0.7$ & $0.64421768723769100$&$ 0.64421768723769080$& $0.6442186501$ 
		\\
		\hline
		$0.8$ & $0.71735609089952280$&$ 0.71735609089952270$& $0.7173567749$ 
		\\
		\hline
		$0.9$ & $0.78332690962748340$&$ 0.78332690962748330$& $0.7833271803$ 
		\\
		\hline
	\end{tabular}
	\label{table Example 4.3solution} 
\end{table}

\begin{table}[H]
	\caption {Comparison of absolute errors of the present method  and those in \cite{Costabile2015}, \cite{Dang2024},  \cite{Hosseini2023},  \cite{Noor2007} for Example \ref{Example 4.3}.}
	\centering 
	\begin{tabular}{|c| c| c |c|c|c|} 
		\hline
		$x$ & SHCM $(N=16)$&Method \cite{Dang2024} &Method \cite{Hosseini2023}&Method \cite{Costabile2015}& Method \cite{Noor2007} $(N=11)$ 
		\\ [0.5ex] 
		\hline 
		$0.1$ & $1.38\times10^{-17}$ & -& $5.11\times10^{-9}$ & $4.45\times10^{-10}$  & $7.79\times10^{-8}$ 
		\\
		\hline
		$0.2$ & $2.7\times10^{-17}$ &-& $3.05\times10^{-8}$ & $5.54\times10^{-10}$  & $2.72\times10^{-7}$ 
		\\
		\hline
		$0.3$ & $5.55\times10^{-17}$ &-& $7.18\times10^{-8}$ & $8.95\times10^{-11}$ &  $5.25\times10^{-7}$ 
		\\
		\hline
		$0.4$ & $5.55\times10^{-17}$ &-& $1.07\times10^{-7}$ & $2.03\times10^{-10}$ &  $7.77\times10^{-7}$ 
		\\
		\hline
		$0.5$ & $0.00000$ &-& $1.22\times10^{-7}$  &$3.32\times10^{-11}$  & $9.71\times10^{-7}$  
		\\
		\hline
		$0.6$ &$2.22\times10^{-16}$  &-&$1.06\times10^{-7}$  & $1.53\times10^{-10}$ & $1.05\times10^{-6}$ 
		\\
		\hline
		$0.7$ & $1.11\times10^{-16}$ &-& $6.08\times10^{-8}$ & $9.48\times10^{-11}$ & $9.63\times10^{-7}$  
		\\
		\hline
		$0.8$ & $1.11\times10^{-16}$ &-& $1.31\times10^{-8}$ & $5.18\times10^{-10}$  & $6.84\times10^{-7}$ 
		\\
		\hline
		$0.9$ & $1.11\times10^{-16}$ &-& $6.35\times10^{-9}$ & $4.15\times10^{-10}$  & $2.71\times10^{-7}$ 
		\\
		\hline
		\hline
		MAE &$2.22\times10^{-16}$ & $3.26\times 10^{-14}(N=16)$ &$1.22\times 10^{-7}$ & $5.54\times 10^{-10}$   & $1.05\times 10^{-6}$  
		\\
	&&$1.29\times 10^{-16}(N=32)$&  
	&&
	\\ [0.5ex] 	
			\hline	
	\end{tabular}
	\label{table Example 4.3error} 
\end{table}

\begin{figure}[h!]
	\begin{center}
		\subfigure[Exact and numerical (SHCM) solutions.]{%
			\label{fig:Example4.3a}
			\includegraphics[width=0.475\textwidth]{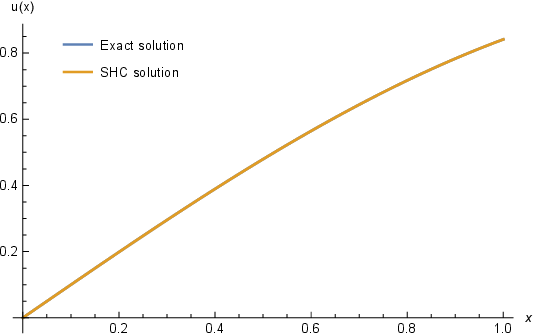}
		}
		\subfigure[Absolute errors of SHCM.]{%
			\label{fig:Example4.3b}
			\includegraphics[width=0.475\textwidth]{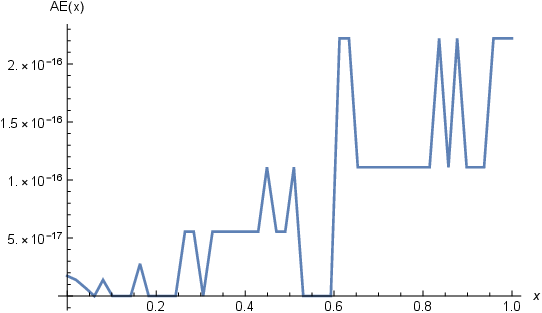}
		}
	\end{center}
	\caption{Comparison of exact and numerical solutions for Example \ref{Example 4.3}.
	}%
	\label{fig:Example 4.3}
\end{figure}

\begin{example}\label{Example 4.4}
Given  the nonlinear boundary value problem (\cite{Dang2024})
	\begin{equation}\label{eqflanbvp 34th-1}
		\begin{split}		
			&u^{(4)}(x)-\frac{1}{8}[u(x)u'''(x)-u'(x)u''(x)] =\pi^{4}\sin \pi x		\qquad (0<x<1)	\\
			&u(0)=0,\quad u(1)=0, \quad u'(0)=\pi, \quad u'(1)=-\pi.
		\end{split}
	\end{equation}
	The exact solution  is $u(x)=\sin \pi x$. Applying the shifted Horadam collocation algorithm in Section \ref{Sec SHCM} to the problem \eqref{eqflanbvp 34th-1} one has the following  expansion coefficients.
	\begin{align}
		c_0=&0.721703, \quad c_1=-1.39538\times 10^{-18}, \quad c_2=-0.263698, \quad c_3=1.30957\times 10^{-18}\nonumber\\
		c_4=&0.0142944,\quad c_5=-4.92535\times 10^{-19}, \quad c_6=-0.0003017, \quad c_7= 4.06408\times 10^{-20}\nonumber\\
		c_8=& 3.37546\times 10^{-6}, \quad c_9=3.86895\times 10^{-20}, \quad c_{10}=-2.33758\times 10^{-8}\nonumber\\ c_{11}=& -2.4645\times 10^{-20}, \quad 
		c_{12}=1.10047\times 10^{-10}, \quad c_{13}=7.82427\times 10^{-21}\nonumber\\ c_{14}=&-3.75048\times 10^{-13}, \quad c_{15}=-1.13349\times 10^{-21}, \quad
		c_{16}=9.61907\times 10^{-16}.	
	\end{align}	
It follows that the approximate shifted Horadam polynomial series solution gives
	\begin{align}
		u(x)=&-3.45153\times 10^{-17}+3.14159 x-2.59574\times 10^{-15} x^2-5.16771 x^3\nonumber\\
		&+4.25152\times 10^{-11} x^4+2.55016 x^5+1.487019\times 10^{-8} x^6-0.599265 x^7\nonumber\\
		&+6.45350\times 10^{-7} x^8+0.0821434 x^9+6.76038\times 10^{-6} x^{10}-0.00738424 x^{11}\nonumber\\
		&+0.0000209969 x^{12}+0.00044265 x^{13}+0.0000193915 x^{14}-0.0000330509 x^{15}\nonumber\\
		&+4.13135\times 10^{-6} x^{16}.
	\end{align}

\end{example}

In Table \ref{table Example 4.4}, we illustrate the comparison of the exact solution, SHC solution, and the method presented in \cite{Dang2024}. It is  observed from the absolute errors obtained that the present method performs better than the method in a recent paper \cite{Dang2024}. In Figure \ref{fig:Example4.4a}, we plot the graphs of the exact and SHC solutions, while Figure \ref{fig:Example4.4b} presents the graph of the absolute errors of the present method.

\begin{table}[H]
	\caption {Comparison of solutions and absolute errors of the present method  and those in \cite{Dang2024} for Example \ref{Example 4.4}.}
	\centering 
	\begin{tabular}{|c| c| c |c|c|} 
		\hline
		$x$ &Exact Solution &Present Method& SHCM $(N=16)$& Dang et al.  \cite{Dang2024} 
		\\ [0.5ex] 
		\hline 
		$0.1$ &$0.30901699437494744$&$ 0.3090169943749475$& $5.55112\times10^{-17}$ &-
		\\
		\hline
		$0.2$ &$0.5877852522924731$&$ 0.5877852522924732$& $1.11022\times10^{-16}$ & - 
		\\
		\hline
		$0.3$ &$0.8090169943749475$&$ 0.8090169943749472$& $2.22045\times10^{-16}$  &- 
		\\
		\hline
		$0.4$ &$0.9510565162951535$&$ 0.9510565162951536$&$1.11022\times10^{-16}$  &- 
		\\
		\hline
		$0.5$ &$1.0000000000000000$&$ 1.0000000000000002$&$2.22045\times10^{-16}$  &- 
		\\
		\hline
		$0.6$ &$0.9510565162951536$&$ 0.9510565162951536$&$0.00000$  & - 
		\\
		\hline
		$0.7$ &$0.8090169943749475$&$ 0.8090169943749478$&$3.33067\times10^{-16}$  &-
		\\
		\hline
		$0.8$ &$0.5877852522924732$&$ 0.5877852522924734$&$2.22045\times10^{-16}$  &-
		\\\hline
		$0.9$ &$0.3090169943749475$&$ 0.3090169943749473$&$1.66533\times10^{-16}$  &- 
		\\	
		\hline
		MAE&&&$3.33067\times10^{-16}$  & $9.3762\times10^{-11}(N=16)$ 
	\\
&&&  
&$3.6743\times 10^{-13}(N=32)$
\\ [0.5ex] 	
&&&  
&$1.4352\times 10^{-15}(N=64)$
\\ [0.5ex] 	
\hline
	\end{tabular}
	\label{table Example 4.4} 
\end{table}

\begin{figure}[h!]
	\begin{center}
		\subfigure[Exact and numerical (SHCM) solutions.]{%
			\label{fig:Example4.4a}
			\includegraphics[width=0.475\textwidth]{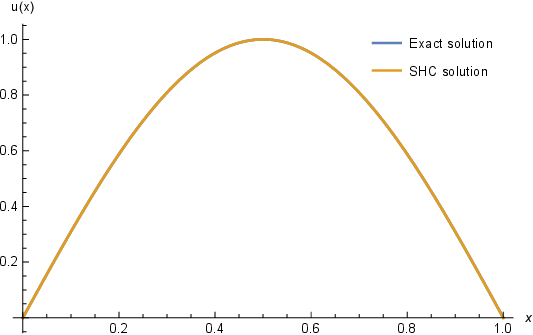}
		}
		\subfigure[Absolute errors of SHCM.]{%
			\label{fig:Example4.4b}
			\includegraphics[width=0.475\textwidth]{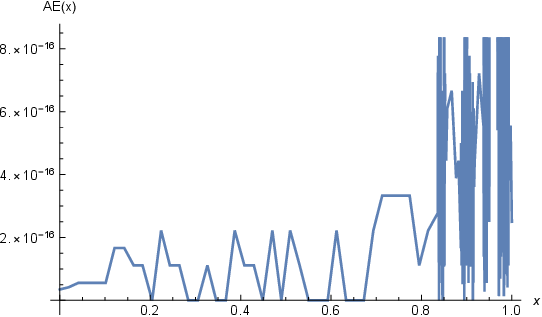}
		}
	\end{center}
	\caption{Comparison of exact and numerical solutions for Example \ref{Example 4.4}.
	}%
	\label{fig:Example 4.4}
\end{figure}

\begin{example}\label{Example 4.5}
Consider the nonlinear boundary value problem (\cite{Dang2024}, \cite{Mohanty2000})
	\begin{equation}\label{eqflanbvp 34th-2}
		\begin{split}		
			&u^{(4)}(x)-u(x)[u'(x)+u''(x)+u'''(x)] =16e^{2x}-14e^{4x} \qquad (0<x<1)	\\
			&u(0)=1,\quad u(1)=e^{2}, \quad u'(0)=2, \quad u'(1)=2e^{2}.
		\end{split}
	\end{equation}
	The exact solution  is $u(x)=e^{2x}$.  Applying the shifted Horadam collocation algorithm in Section \ref{Sec SHCM} to the problem \eqref{eqflanbvp 34th-1} one has the following  expansion coefficients.
	\begin{align}
		c_0=&3.07252, \quad c_1=-1.47600\times 10^{-18}, \quad c_2= 0.36156, \quad c_3=-0.0595221\nonumber\\
		c_4=&0.00737913,\quad c_5=-0.000733566, \quad c_6=0.0000608598, \quad c_7= -4.33212\times 10^{-6}\nonumber\\
		c_8=&2.70009\times 10^{-7}, \quad c_9=-1.49665\times 10^{-8}, \quad c_{10}=7.46916\times 10^{-10}\nonumber\\ c_{11}=& -3.38965\times 10^{-11}, \quad 
		c_{12}=1.41041\times 10^{-12}, \quad c_{13}=-5.41819\times 10^{-14}\nonumber\\ c_{14}=&1.93304\times 10^{-15}, \quad c_{15}=-6.46686\times 10^{-17}, \quad
		c_{16}=2.02037\times 10^{-18}.	
	\end{align}	
Hence, we obtain the approximate series solution 
	\begin{align}
		u(x)=&1.00000+2.00000 x+2. x^2+1.33333 x^3+0.666667 x^4+0.266667 x^5+0.0888889 x^6\nonumber\\
		&+0.0253968 x^7+0.00634921 x^8+0.00141091 x^9+0.000282237 x^{10}\nonumber\\
		&+0.0000512186 x^{11}+8.66243\times 10^{-6} x^{12}+1.21529\times 10^{-6} x^{13}+2.50305\times 10^{-7} x^{14}\nonumber\\
		&+1.80762\times 10^{-11} x^{15}+8.67742\times 10^{-9} x^{16}.
	\end{align}

\end{example}

Table \ref{table Example 4.5} demonstrates the comparison of the exact solution, SHC solution, and the absolute errors of the present method, and the methods presented in \cite{Dang2024} and \cite{Mohanty2000}. One clearly sees that the maximum absolute error of the present method is smaller than the ones in \cite{Mohanty2000} and  a recent paper \cite{Dang2024}. The graphs of the comparison of the exact and SHC solutions are plotted in Figure \ref{fig:Example4.5a} and the graph of the absolute errors is shown in Figure \ref{fig:Example4.5b}.

\begin{landscape}
\begin{table}[H]
	\caption {Comparison of absolute errors of the present method  and those in \cite{Dang2024}, \cite{Mohanty2000} 
		for Example \ref{Example 4.5}.}
	\centering 
	\begin{tabular}{|c| c| c |c|c|c|} 
		\hline
		$x$ &Exact Solution&Present Solution&  SHCM $(N=16)$&Dang et al. \cite{Dang2024} &Mohanty \cite{Mohanty2000}  
		\\ [0.5ex] 
		\hline 
		$0.1$ &$1.2214027581601699$&$ 1.2214027581601696$& $1.54076\times10^{-16}$ &-&- 
		\\
		\hline
		$0.2$ &$1.4918246976412703$&$ 1.4918246976412703$& $4.13656\times10^{-17}$ & -&- 
		\\
		\hline
		$0.3$ &$1.8221188003905089$&$ 1.8221188003905089$& $9.06793\times10^{-17}$  &-&- 
		\\
		\hline
		$0.4$ &$2.2255409284924680$&$ 2.2255409284924683$&$9.48465\times10^{-17}$  &-&- 
		\\
		\hline
		$0.5$ &$2.7182818284590450$&$ 2.7182818284590460$&$3.91848\times10^{-16}$  &- &-
		\\
		\hline
		$0.6$ &$3.3201169227365472$&$ 3.3201169227365477$&$8.21202\times10^{-16}$  & -&- 
		\\
		\hline
		$0.7$ &$4.0551999668446745$&$ 4.0551999668446745$&$2.33266\times10^{-16}$  &-&-
		\\
		\hline
		$0.8$ &$4.9530324243951150$&$ 4.9530324243951160$&$1.35460\times10^{-15}$  &-&-
		\\\hline
		$0.9$ &$6.0496474644129465$&$ 6.0496474644129465$&$8.02310\times10^{-16}$  &-&- 
		\\	
		\hline
		MAE&&&$1.35460\times10^{-15}$  & $7.9444\times10^{-7}(N=16)$ & $3. 1510\times10^{-6}(N=16)$   
	\\
&&&&  
$6.3686\times 10^{-8}(N=32)$&$2.0070\times 10^{-7}(N=32)$
\\ [0.5ex] 	
&&&&$5.1595\times 10^{-9}(N=64)$  
&
\\ [0.5ex] 	
&&&&$3.2485\times 10^{-10}(N=128)$  
&
\\ [0.5ex] 	
\hline	
\end{tabular}
	\label{table Example 4.5} 
\end{table}
\end{landscape}

\begin{figure}[h!]
	\begin{center}
		\subfigure[Exact and numerical (SHCM) solutions.]{%
			\label{fig:Example4.5a}
			\includegraphics[width=0.475\textwidth]{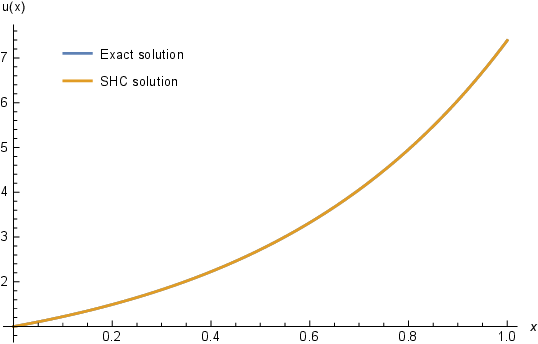}
		}
		\subfigure[Absolute errors of SHCM.]{%
			\label{fig:Example4.5b}
			\includegraphics[width=0.475\textwidth]{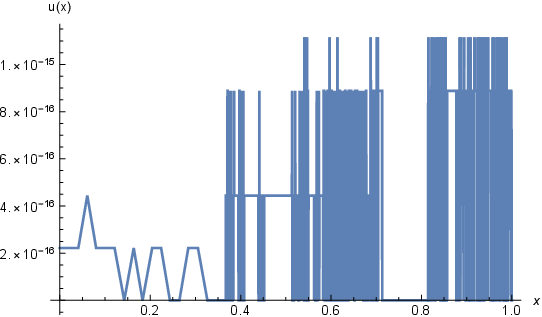}
		}
	\end{center}
	\caption{Comparison of exact and numerical solutions for Example \ref{Example 4.5}.
	}%
	\label{fig:Example 4.5}
\end{figure}

\begin{example}\label{Example 4.6}
Consider the nonlinear boundary value problem (\cite{Dang2024},  \cite{Hajji2008}, \cite{Khna2021}, \cite{Modebei2020}, \cite{Mustafa2017}, \cite{Khan2012bvp})
	\begin{equation}\label{eqflanbvp 34th-3}
		\begin{split}		
			&u^{(4)}(x)-6e^{-4u(x)}=-12(1+x)^{-4}  \qquad (0<x<1) \\
			&u(0)=0,\quad u(1)=\ln 2, \quad u'(0)=1, \quad u'(1)=1/2.
		\end{split}
	\end{equation}
	The exact solution  is $u(x)=\ln(1+x)$. Indeed, specialising the shifted Horadam collocation algorithm in Section \ref{Sec SHCM} to the problem \eqref{eqflanbvp 34th-3}, we obtain the  expansion coefficients
	\begin{align}
		c_0=&0.391171, \quad c_1=-0.169889, \quad c_2=  -0.014502, \quad c_3=-0.00165381\nonumber\\
		c_4=&-0.000212386,\quad c_5=-0.0000291101, \quad c_6= -4.15762\times 10^{-6}, \quad c_7= -6.10919\times 10^{-7}\nonumber\\
		c_8=&-9.16535\times 10^{-8}, \quad c_9=-1.39703\times 10^{-8}, \quad c_{10}= -2.15626\times 10^{-9}\nonumber\\ c_{11}=& -3.36235\times 10^{-10}, \quad 
		c_{12}= -5.2849\times 10^{-11}, \quad c_{13}=-8.31468\times 10^{-12}\nonumber\\ c_{14}=&-1.33251\times 10^{-12}, \quad c_{15}= -2.42543\times 10^{-13}, \quad
		c_{16}=-3.73346\times 10^{-14}.	
	\end{align}	
	Hence, we obtain the approximate series solution 
	\begin{align}
		u(x)=&-1.41971\times 10^{-17}+1. x-0.50000 x^2+0.333333 x^3-0.25000 x^4+0.199998 x^5\nonumber\\
		&-0.166638 x^6+0.142633 x^7-0.123814 x^8+0.106658 x^9-0.087677 x^{10}\nonumber\\
		&+0.0649674 x^{11}-0.0405681 x^{12}+0.0198425 x^{13}-0.00697111 x^{14}+0.00154324 x^{15}\nonumber\\
		&-0.000160351 x^{16}.
	\end{align}
\end{example}

Table \ref{table Example 4.6solution} illustrates the comparison of the exact solution, SHC solution, and the subdivision method in \cite{Mustafa2017}. In Table \ref{table Example 4.6}, we present the comparison of absolute errors of the present method, and the methods presented in  \cite{Dang2024},  \cite{Hajji2008}, \cite{Khna2021}, \cite{Khan2012bvp}, \cite{Mustafa2017}. It is clearly observed  that the present method gives the smallest maximum absolute error.  In Figure \ref{fig:Example4.6a}, we plot the graphs of the exact and SHC solutions, while Figure \ref{fig:Example4.6b} presents the graph of the absolute errors of the present method.

\begin{table}[H]
	\caption {Comparison of exact solution, SHC solution, and solution in  \cite{Mustafa2017} for Example \ref{Example 4.6}.}
	\centering 
	\begin{tabular}{|c| c| c |c|} 
		\hline
		$x$ &Exact Solution& SHC Solution &Method  \cite{Mustafa2017}
		\\ [0.5ex] 
		\hline 
		$0.1$ & $0.0953101798043249$&$ 0.0953101798043314$&$0.0950147533$
		\\
		\hline
		$0.2$ & $0.1823215567939546$&$ 0.1823215567940249$&$0.1814496227$
		\\
		\hline
		$0.3$ & $0.2623642644674911$&$ 0.2623642644674305$&$0.2609546573$
		\\
		\hline
		$0.4$ &$0.3364722366212129$&$ 0.3364722366212644$&$0.3347370220$
		\\
		\hline
		$0.5$ &$0.4054651081081644$&$ 0.4054651081081879$&$0.4036840381$
		\\
		\hline
		$0.6$ &$0.4700036292457356$&$ 0.4700036292456766$&$0.4684459279$
		\\
		\hline
		$0.7$ &$0.5306282510621704$&$ 0.5306282510622377$& $ 0.5294932609$
		\\
		\hline
		$0.8$ &$0.5877866649021191$&$ 0.5877866649020731$&$0.5871580370$
		\\\hline
		$0.9$ &$0.6418538861723947$&$ 0.6418538861723968$ &$ 0.6416636708$
		\\	
		\hline
	\end{tabular}
	\label{table Example 4.6solution} 
\end{table}

\begin{landscape}
	\begin{table}[H]
	\caption {Comparison of absolute errors of the present method  and those in \cite{Dang2024},  \cite{Hajji2008}, \cite{Khna2021}, \cite{Khan2012bvp}, \cite{Mustafa2017} for Example \ref{Example 4.6}.}
	\centering 
	\begin{tabular}{|c| c| c |c|c|c|c|} 
		\hline
		$x$ &  SHCM $(N=16)$& Method \cite{Dang2024} &Method \cite{Khna2021}& Method \cite{Khan2012bvp} & Method \cite{Hajji2008}& Method \cite{Mustafa2017}
		\\ [0.5ex] 
		\hline 
		$0.1$ & $6.46705\times10^{-15}$ &-& -  & - &-&$ 2.954\times10^{-4}$
		\\
		\hline
		$0.2$ & $7.03326\times10^{-16}$ & - & -  &-& $2.438\times10^{-11}$&$8.719\times10^{-4}$
		\\
		\hline
		$0.3$ & $6.05627\times10^{-14}$  &- &-   & -&-&$1.409\times10^{-3}$
		\\
		\hline
		$0.4$ &$5.14588\times10^{-14}$  &- &  - &-&$3.943\times10^{-12}$&$1.735\times10^{-3}$
		\\
		\hline
		$0.5$ &$2.35367\times10^{-14}$  &- &-  & -&-&$ 1.781\times10^{-3}$
		\\
		\hline
		$0.6$ &$5.90639\times10^{-14}$  & - &-  & -&$3.274\times10^{-12}$&$1.557\times10^{-3}$
		\\
		\hline
		$0.7$ &$6.72795\times10^{-14}$  &-& -  & -&-&$ 1.134\times10^{-3}$
		\\
		\hline
		$0.8$ &$4.60743\times10^{-14}$  &-& -  & -&$4.789\times10^{-11}$&$6.286\times10^{-4}$
		\\\hline
		$0.9$ &$2.10942\times10^{-14}$  &- &-  & -&-&$ 1.902\times10^{-4}$
		\\	
		\hline
		MAE&$6.72795\times10^{-14}$  & $3.8797\times10^{-11}(N=16)$ & $2. 44\times10^{-10}(N=16)$   & $9.39\times\times10^{-12}$&$4.789\times10^{-11}$&$1.781\times10^{-3}$
		\\
	&&$1.7517\times 10^{-13}(N=32)$&$1.10\times 10^{-11}(N=32)$&&  
	&$2.0070\times 10^{-7}(N=32)$
	\\ [0.5ex] 	
	&&&$1.51\times 10^{-12}(N=64)$&&  
	&
	\\ [0.5ex] 	
	\hline
	\end{tabular}
	\label{table Example 4.6} 
\end{table}

\end{landscape}

\begin{figure}[h!]
	\begin{center}
		\subfigure[Exact and numerical (SHCM) solutions.]{%
			\label{fig:Example4.6a}
			\includegraphics[width=0.475\textwidth]{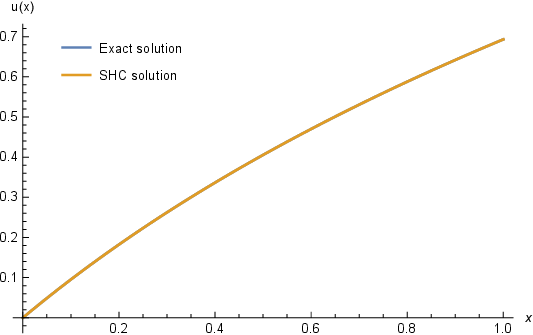}
		}
		\subfigure[Absolute errors of SHCM.]{%
			\label{fig:Example4.6b}
			\includegraphics[width=0.475\textwidth]{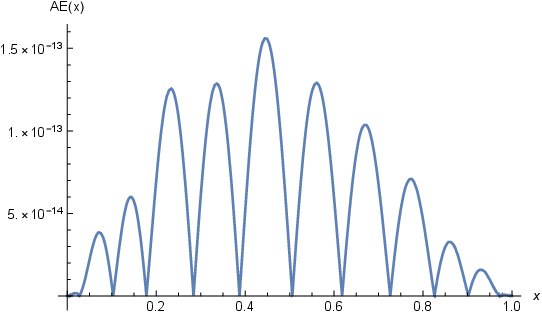}
		}
	\end{center}
	\caption{Comparison of exact and numerical solutions for Example \ref{Example 4.6}.
	}%
	\label{fig:Example 4.6}
\end{figure}

\begin{example}\label{Example 4.7}
Given the nonlinear boundary value problem (\cite{Dang2024}, \cite{Khna2021})
	\begin{equation}\label{eqflanbvp 34th-4}
		\begin{split}		
			&u^{(4)}(x)+u^{2}(x)=-8x\cos x+(x^{2}-13)\sin x+(x^{4}-2x^{2}+1)\sin^{2}x \qquad (0<x<1)\\
			&u(0)=0,\quad u(1)=0, \quad u'(0)=-1, \quad u'(1)=2\sin 1.
		\end{split}
	\end{equation}
	The exact solution  is $u(x)=(x^{2}-1)\sin x$.  Indeed, specialising the shifted Horadam collocation algorithm in Section \ref{Sec SHCM} to the problem \eqref{eqflanbvp 34th-3}, we obtain the  expansion coefficients
	\begin{align}
		c_0=&-0.266618, \quad c_1=0.0219665, \quad c_2=  0.0914115, \quad c_3=-0.0111718\nonumber\\
		c_4=&-0.00153258,\quad c_5=0.000126219, \quad c_6= 6.69415\times 10^{-6}, \quad c_7=  -4.08709\times 10^{-7}\nonumber\\
		c_8=&-1.28302\times 10^{-8}, \quad c_9=6.20319\times 10^{-10}, \quad c_{10}= 1.36393\times 10^{-11}\nonumber\\ c_{11}=& -5.4526\times 10^{-13}, \quad 
		c_{12}= -9.1746\times 10^{-15}, \quad c_{13}=3.12499\times 10^{-16}\nonumber\\ c_{14}=&4.24735\times 10^{-18}, \quad c_{15}= -1.25789\times 10^{-19}, \quad
		c_{16}=-1.43422\times 10^{-21}.	
	\end{align}	
Thus we obtain the approximate series solution 
	\begin{align}
		u(x)=&8.32437\times 10^{-18}-1.00000 x-6.62955\times 10^{-17} x^2+1.16667 x^3-3.48842\times 10^{-16} x^4\nonumber\\
		&-0.17500 x^5-1.16128\times 10^{-13} x^6+0.00853175 x^7-4.21472\times 10^{-12} x^8\nonumber\\
		&-0.000201168 x^9-3.56319\times 10^{-11} x^{10}+2.78084\times 10^{-6} x^{11}-8.32452\times 10^{-11} x^{12}\nonumber\\
		&-2.51343\times 10^{-8} x^{13}-5.18721\times 10^{-11} x^{14}+1.84344\times 10^{-10} x^{15}\nonumber\\
		&-6.15991\times 10^{-12} x^{16}.
	\end{align}
\end{example}

Table \ref{table Example 4.7solution} presents the comparison of the exact solution and SHC solution for Example \ref{Example 4.7}. In Table \ref{table Example 4.7}, we demonstrate the comparison of absolute errors of the present method, and the methods presented in \cite{Dang2024}, \cite{Khna2021}, \cite{Zahra2011} for  Example \ref{Example 4.7}. It is seen clearly that the present method outperforms the other existing methods under comparison. The graphs of the comparison of the exact and SHC solutions are plotted in Figure \ref{fig:Example4.7a} and the graph of the absolute errors is shown in Figure \ref{fig:Example4.7b}.

\begin{table}[H]
	\caption {Comparison of exact solution and SHC solution for Example \ref{Example 4.6}.}
	\centering 
	\begin{tabular}{|c| c| c |} 
		\hline
		$x$ &Exact Solution& SHC Solution 
		\\ [0.5ex] 
		\hline 
		$0.1$ &$-0.09883508248035987$&$ -0.09883508248035985$
		\\
		\hline
		$0.2$ & $-0.19072255756325876$&$ -0.19072255756325873$
		\\
		\hline
		$0.3$ & $-0.26892338806181900$&$ -0.26892338806181887$
		\\
		\hline
		$0.4$ &$-0.32711140753926643$&$ -0.32711140753926630$
		\\
		\hline
		$0.5$ &$-0.35956915395315225$&$ -0.35956915395315220$
		\\
		\hline
		$0.6$ &$-0.36137118297282267$&$ -0.36137118297282245$
		\\
		\hline
		$0.7$ &$-0.32855102049122240$&$ -0.32855102049122230$
		\\
		\hline
		$0.8$ &$-0.25824819272382810$&$ -0.25824819272382790$
		\\\hline
		$0.9$ &$-0.14883211282922182$&$ -0.14883211282922143$ 
		\\	
		\hline
	\end{tabular}
	\label{table Example 4.7solution} 
\end{table}

\begin{table}[H]
	\caption {Comparison of absolute errors of the present method  and those in \cite{Dang2024}, \cite{Khna2021}, \cite{Zahra2011} for  Example \ref{Example 4.7}.}
	\centering 
	\begin{tabular}{|c| c| c |c|c|} 
		\hline
		$x$ & SHCM $(N=16)$& Method \cite{Dang2024}&Method \cite{Khna2021} &Method  \cite{Zahra2011} 
		\\ [0.5ex] 
		\hline 
		$0.1$ & $1.38778\times 10^{-17}$&- & -  &- 
		\\
		\hline
		$0.2$ & $2.77556\times 10^{-17}$& - & -  & - 
		\\
		\hline
		$0.3$ & $1.11022\times 10^{-16}$& -  &- &- 
		\\
		\hline
		$0.4$ &$1.11022\times 10^{-16}$&-& -  &- 
		\\
		\hline
		$0.5$ & $5.55112\times 10^{-17}$&-& - & - 
		\\
		\hline
		$0.6$ & $2.22045\times 10^{-16}$&-&-& - 
		\\
		\hline
		$0.7$ & $5.55112\times 10^{-17}$&-  &-  & -  
		\\
		\hline
		$0.8$ &$2.22045\times 10^{-16}$&-  & - & - 
		\\
		\hline
		$0.9$ & $3.88578\times 10^{-16}$&- & - &- 
		\\
		\hline\hline
		MAE & $3.88578\times 10^{-16}$&$1.0000\times 10^{-13}(N=16)$ & $8.51\times 10^{-12}(N=16)$  & $ 5.83355\times 10^{-12}$ 
				\\
		&&$3.9239\times 10^{-16}(N=32)$&$1.36\times 10^{-12}(N=32)$  
		&
		\\ [0.5ex] 	
		&&&  $9.97\times 10^{-13}(N=64)$
		&
		\\ [0.5ex] 	
			\hline
	\end{tabular}
	\label{table Example 4.7} 
\end{table}

\begin{figure}[h!]
	\begin{center}
		\subfigure[Exact and numerical (SHCM) solutions.]{%
			\label{fig:Example4.7a}
			\includegraphics[width=0.475\textwidth]{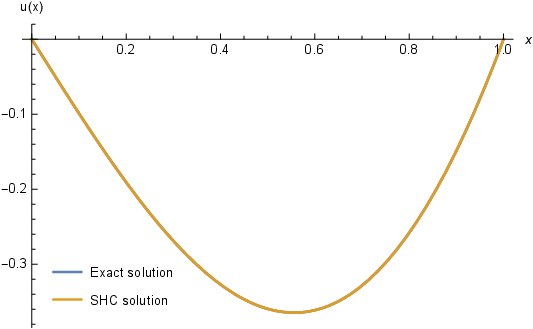}
		}
		\subfigure[Absolute errors of SHCM.]{%
			\label{fig:Example4.7b}
			\includegraphics[width=0.475\textwidth]{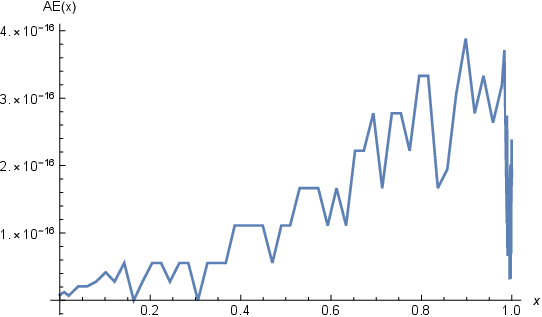}
		}
	\end{center}
	\caption{Comparison of exact and numerical solutions for Example \ref{Example 4.7}.
	}%
	\label{fig:Example 4.7}
\end{figure}

%
%

%
%


%

\section{Conclusion}

 In this paper, an efficient collocation method based on shifted Horadam polynomials has been applied to obtain numerical solutions of a nonlinear fourth-order boundary value problem arising in beam theory. Seven examples of the proposed problem were considered to illustrate the efficiency, reliability, and accuracy of the shifted Horadam collocation method. For comparison purposes, we considered known examples from published papers. Six of the seven examples were considered in a recent paper of Dang et al. \cite{Dang2024} using different methods. In each of the examples considered, the  numerical solutions obtained using the present method were compared with the exact solution and other existing results. It was interestingly clear that, for all the examples considered, the approximate solutions, absolute errors, and maximum absolute errors obtained using the present method outperformed the existing methods under comparison. Hence, the present method was found to be better than the methods considered in \cite{Costabile2015},  \cite{Dang2024}, \cite{Hajji2008}, 
 \cite{Hosseini2023}, \cite{Khna2021}, \cite{Khan2012bvp}, \cite{Moghadam2022}, \cite{Mohanty2000}, \cite{Mustafa2017},  
 \cite{Noor2007},   \cite{Singh2014},     \cite{Zahra2011}. The graphs of the exact and approximate solutions, as well as those of the absolute errors were plotted.

	%
	%


\paragraph{Conflicts of Interest} The author declares that there is no conflict of interest regarding the publication of this paper.		

\paragraph{Funding}  The author did not receive any fund for the research presented in this paper. 


\paragraph{Data Availability}  Not applicable.


\begin{thebibliography}{999}
		
		



\bibitem{AbdElhameed2025} W.M. Abd-Elhameed, O. M. Alqubori, A.G. Atta, A collocation approach for the nonlinear fifth-order KdV equations using certain shifted Horadam polynomials, \textit{Mathematics}, 13, 2025, 300, \url{doi:10.3390/math13020300} 


\bibitem{Adak2021} M. Adak and A. Mandal, Numerical Solution of Fourth-Order Boundary Value Problems for Euler-Bernoulli Beam Equation
using FDM, \textit{J. Phys.: Conf. Ser.},  2070, 2021, 012052.

\bibitem{Agarwal1986} R.P Agarwal, Boundary value problems from higher order differential equations, World Scientific, 1986.

\bibitem{AgarwalR2023} R. Agarwal, G. Mihaylova, P. Kelevedjiev, Existence for Nonlinear Fourth-Order Two-Point Boundary
Value Problems, \textit{Dynamics},  2023, 3, 152–170.


\bibitem{Ahsan2023} M. Ahsan, M. Bohner, A. Ullah, A. A. Khan, S. Ahmad, A Haar wavelet multi-resolution collocation method for singularly perturbed differential equations with integral boundary conditions, \textit{Math. Comput. Simul.}, \textbf{204}, 2023, 166-180.


\bibitem{Akinnukawe2024}  B.I. Akinnukawe, J.O. Kuboye, S.A. Okunuga,   Numerical solution of fourth-order initial value
problems using novel fourth-order block algorithm, \textit{J. Nepal Math. Soc.},  6, 2024,  7-18. 


\bibitem{Akram2013} G. Akram, H.U. Rehman,  Numerical solution of eight order boundary value problems in reproducing kernel space, \textit{Numer. Algor.}, 2013, 62, 527-540.

\bibitem{Alharbi2024} M.H. Alharbi, A.F. Abu Sunayh, A.G. Atta, W.M. Abd-Elhameed, Nover approach by shifted Fibonacci polynomials for solving the fractional Burgers equation, \textit{Fractal Fract}, 8, 2024, 427. 

\bibitem{Ali2019} H.S. Ali, E. Alali, A. Ebaid, F.M. Alharbi, Analytical solution of a class of singular second order boundary value problems with applications, \textit{Math.}, 7, 2019, 1-10.



\bibitem{Ali2011} A. Ali, F. Haq,  Numerical solution of fourth order boundary value problems using Haar wavelets, \textit{Appl. Math. Sci.}, 5, 2011, 3131-3146.

\bibitem{AminRo2021} R. Amin, K. Shah, Q.M. Al-Mdallal, I. Khan,  Efficient numerical algorithm for the solution of eight order boundary value problems by Haar wavet method, \textit{Int. Appl. Comput. Math.}, 2021, 7:34.  \url{doi:10.1007/s40819-021-00975-x}




\bibitem{AwoEmdenFowler} Awonusika RO. Analytical solutions of generalised Emden–Fowler initial and
boundary value problems of higher order. {\it Int J Appl Comput Math}. 2024;10:43.  \url{https://doi.org/10.1007/s40819-024-01676-x} 

\bibitem{AwoOnuoha} Awonusika RO, Onuoha OB. Analytical method for systems of nonlinear   
singular boundary value problems. \textit{Part. Differ. Equat. Appl. Math.}, 2024:11;100762


\bibitem{Bai2007} Bai, Z. The upper and lower solution method for some fourth-order boundary value problems. \textit{Nonlinear Anal. Theory Methods
	Appl.} 2007, 67, 1704–1709.


\bibitem{Benaicha2016} Benaicha, S.; Haddouchi, F. Positive solutions of a nonlinear fourth-order integral boundary value problem.\textit{ Ann. West Univ. Timis.-Math. Comput. Sci.} 2016, 54, 73–86

\bibitem{Bishop1989} R.E.D. Bishop, S.M. Cannon, S. Miao, On coupled bending and torsional vibration of uniform beams, \textit{J. sound Vib.}, 131, 1989, 309-325.


\bibitem{Canuto1991} Canuto, C.; Hussaini, M.; Quarteroni, A.; Zang, T. Spectral Methods in Fluid Dynamics; Springer Series in Computational Physics; Springer: Berlin/Heidelberg, Germany, 1991.


\bibitem{Costabile2015} F.A. Costabile, A. Napoli,  Collocation for high order differential equations with  two-points Hermite boundary conditions, \textit{Appl. Numer. Math.}, 2015, 87, 157-167.

\bibitem{Dang2010} Dang,Q.A.; Luan, V.T. Iterative method for solving a nonlinear fourth order boundary value problem. \textit{Comput. Math. Appl.} 2010, 60, 112–121.

\bibitem{Dang2020}  Dang,Q.A.; Dang, Q.L. Existence results and iterative method for a fully fourth-order nonlinear integral boundary value problem. \textit{Numer. Algorithms}, 2020, 85, 887–907

		

\bibitem{Dang2024} Q. A. Dang, T. H. Nguyen, V. Q. Vu,  Construction of high order numerical methods for solving fourth order nonlinear boundary value problems, \textit{Numer. Algor.}, 2024,  \url{doi:10.1007/s11075-024-01879-9} 


\bibitem{Davidson2006} F.A. Davidson, B.P. Ryne, The formulation of second-order boundary value problems on time scales, \textit{Adv. Diff. Equat.}, 2006, 1-0.

\bibitem{Dimitrov2024} N. D. , J. M. Jonnalagadda, Existence and Nonexistence Results for a Fourth-Order Boundary Value Problem with Sign-Changing Green’s Function, \textit{Mathematics},  2024, 12, 2456.


\bibitem{Erturk2007} V. S. Ert\"{u}rk, S. Momani,  Comparing numerical methods for solving fourth-order boundary value problems, \textit{Applied Mathematics and Computation}, 188 (2007) 1963–1968.

\bibitem{Gul2022} M. Gul, H. Khan, A. Ali, The solution of fifth and sixth order linear and nonlinear boundary values problems by the improved residual power series method, \textit{J. Math. Anal. Modeling}, \textbf{3}, 2022, 1-14.	


\bibitem{Hajji2008} M. A. Hajji, K. Al-Khaled,
Numerical methods for nonlinear
fourth-order boundary value problems with applications, \textit{International Journal of Computer
	Mathematics}, 85, 2008, 83–104.

\bibitem{Hauser2009} J. R. Hauser, Numerical Methods for Nonlinear Engineering Models, Springer, 2009, 883-987.

\bibitem{Horadam1996} A.F. Horadam, Extension of a synthesis for a class of polynomial sequences, \textit{Fibonacci Q.}, 34, 1996, 68-74. 

\bibitem{Horadam1997} Horadam, A.F.: Jacobsthal representation polynomials. \textit{Fibonacci Quart}. 35, 137–148 (1997).

\bibitem{Horadam1985} Horadam, A.F., Mahon, J.M.: Pell and Pell–Lucas polynomials. \textit{Fibonacci Quart}. 23, 7–20 (1985).




\bibitem{Hossain2014} Md. B. Hossain, Md. S. Islam,  Numerical solutions of general fourth order two pint boundary value problems Galerkin method with Legendre polynomials, \textit{Dhaka Univ. J. Sci.}, 62, 2014, 103-108.

\bibitem{Hosseini2023} S.S. Hosseini, A. Aminataei, F. Kianya, M. Alizadeha, M. Zahraei, Change in the form of fourth order two-point boundary value problem for solving by Adomian decomposition and homotopy perturbation methods, \textit{Int. J. Nonlinear Anal. Appl.}, 14 (2023) 7, 255–260.

\bibitem{IqbalMK2020} M. K. Iqbal, New quartic B-spline approximations for numerical solution of fourth order singular  boundary value problems, \textit{Punjab Univ. J. Math.}, \textbf{52}, 2020, 47-63.

\bibitem{Islam2010} S.U. Islam, I. Aziz, B. Sarler, The numerical solution of second-order boundary-value problems by collocation method with Haar wavelets, \textit{Math. Comm. Mode.}, 52, 2010, 1577-1590.


\bibitem{Kelesoglu2014} O. Kelesoglu,  The Solution of Fourth Order Boundary Value Problem Arising out of the Beam-Column Theory Using Adomian Decomposition Method,  \textit{Mathematical Problems in Engineering}, 2014, 2014, 1-6.

\bibitem{Khan2003} A. Khan, T. Aziz, Parametric cubic spline approach to the solution of a system of second order boundary value problems, \textit{J. Optim. Theory Appl.}, 118, 2003, 45-54.

\bibitem{Khna2021} A. Khan, S. Bisht, Exponential spline solution of boundary value problems occurring in the plate deflection theory, \textit{Proc. Natl. Acad. Sci.}, 91, 2021, 289-295. 


		
\bibitem{Khan2024bvp} I. Khan, K.J. Ansari, R.A. Amin, Sundas, H. Farheen, Application of linear legendre multi-wavelets collocation method for solution of fourth order boundary value problems, \textit{Eng. Comput.}, 2024.  \url{doi:10.1007/s00366-024-02078-9}  

\bibitem{Khan2012bvp} Khan A., Zahra W.K., Khandelwal P., Non-polynomial septic splines approach to the solution of fourth-order two point
boundary value problems. \textit{Int J Nonlinear Sci} 13(3):363–372.


\bibitem{KumarN2023} N. Kumar, D. Tiwari, A. K. Verma, C. Cattani, Hybrid model for the optimal numerical solution of nonlinear ordinary differential systems, \textit{Comput. Appl. Math.}, \textbf{42}, 2023:322. \url{http://dx.doi.org/10.1007/s4314-023-02468-7}

\bibitem{Leeb2020} W. Leeb, V. Rokhlin, On the numerical solution of fourth-order linear two-oint boundary value problems, \textit{SIAM J. Sci. Comput.}, 2020,  42, A1789-A1808.


\bibitem{Lewis2022} B.J. Lewis, E.N. Onder, A.A. Prudil, Advanced Mathematics for Engineering Students, Butterworth-Heinemann, 2022.



\bibitem{Li2013} Li, H.; Wang, L.; Pei, M. Solvability of a fourth-order boundary value problem with integral boundary conditions. \textit{J. Appl. Math.}, 2013, 782363, 1–7.

\bibitem{Lv2015} Lv, X.; Wang, L.; Pei, M. Monotone positive solution of a fourth-order BVP with integral boundary conditions. \textit{Bound. Value Probl.} 2015, 2015, 172.

\bibitem{Mehrpouya2024}  M.A. Mehrpouya, R. Salehi, P.J.Y. Wong,  A fast and accurate numerical method for solving Nnonlinear
fourth-order boundary value problems in the beam theory, \textit{ Axioms},  13, 2024,  757. \url{https://doi.org/10.3390/axioms13110757}



\bibitem{Modebei2020} M.I. Modebei, S.N. Jator, H. Ramos, Block hybrid method for the numerical solution of fourth order boundary value problem, \textit{J. Comput. Appl. Math.}, 377, 2020, 1-15. 


\bibitem{Moghadam2022} A.A. Moghadam, A.R. Soheili, A.S. Bagherzadeh, Numerical solution of fourth-order BVPs by using Lidstone-collocation method, \textit{Appl. Math. Comput.}, 425, 2022, 127055. 


\bibitem{Mohanty2000} R.K. Mohanty, A fourth-order finite difference method for the general one-dimensional nonlinear biharmonic problems of the first kind, \textit{J. Comput. Appl. Math.}, 114, 2000, 275-290. 

\bibitem{Momani2006} S. Momani, K. Moadi, A reliable algorithm for solving fourth-order boundary value problems, \textit{J. Appl. Math. Comput.} 22 (3) (2006) 185–197.

\bibitem{Mustafa2017} G.Mustafa, M.Abbas,S.T.Ejaz, A.I. M.Ismail, andF. Khan, A numerical approach based on subdivision schemes for solving
non-linear fourth order boundary value problems, \textit{Journal of Computational Analysis and Applications}, 23, 607-623, 2017.


\bibitem{Noor2007} M.A. Noor, S.T. Mohyud-Din, An efficient method for fourth-order boundary value problems, \textit{Comput . Math. Appl.}, 2007, 54, 1101-1111.


\bibitem{Palamides2012} P.K. Palamides, A.P. Palamides, Fourth-Order Four-Point Boundary Value Problem: A Solutions Funnel Approach, International Journal of Mathematics and Mathematical Sciences, 2012, 2012, 1-18.

\bibitem{PanditB2021} B. Pandit, A. K. Verma, R. P. Agarwal, Numerical approximations for a class of nonlinear higher order singular boundary value problems by using homotopy perturbation and variational iteration method, \textit{Comput. Math. Meth.}, \textbf{3}, 2021, e1195.
%
%
%
\bibitem{PrantaSSD2023} S. S. D. Pranta, M. S. Islam, Numerical approximations of a class of nonlinear second-order boundary value problems using Galerkin-compact finite difference method, \textit{European J. Math. Stat.}, \textbf{4}, 2023, 56-68.

\bibitem{Qayyum2023} M. Qayyum, Q. Fatima, S. T. Saeed, A. Akg\"{u}l, W. Weera, W. R. Alharbi, A reliable algorithm for higher order boundary value problems, \textit{Alexandria Eng. J.}, \textbf{66}, 2023, 315-328.


\bibitem{Ra02024} R. Rao, J. M. Jonnalagadda, Existence of a unique solution to a fourth-order boundary value problem
and elastic beam analysis, \textit{Mathematical Modelling and Control},  4: 297–306.

\bibitem{Ravindra2024}  S. Yardımcı E. U\v{g}urlu, Nonlinear fourth order boundary value problem, \textit{Boundary Value Problems},  2014, 2014:189, 2014:189


\bibitem{Shahni2023} J. Shahni, R. Singh, C. Cattani, Bernoulli collocation method for the third-order Lane-Emden-Fowler boundary value problem, \textit{Appl. Numer. Math.}, \textbf{186}, 2023, 100-113.

\bibitem{Shahni2023b} J. Shahni, R. Singh, C. Cattani, An efficient numerical approach for solving three-point
Lane-Emden-Fowler boundary value problem, \textit{Math. Computer Simul.}, \textbf{210}, 2023, 1-16.
%

\bibitem{SinghR2020} R. Singh, V. Guleria, M. Singh, Haar wavelet quasilinearization method for numerical solution of Emden-Fowler type equations, \textit{Math. Comput. Simul.}, \textbf{174}, 2020, 123-133.

\bibitem{Singh2014} R. Singh, J. Kumar, G. Nelakanti, Approximate series solution of fourth-order boundary value problems using decomposition method with Green's function, , \textit{J. Math. Chem.}, 52, 2014, 1099-1118. 

\bibitem{Swati2020} K. S. Swati, A. K. Verma, M. Singh, Higher order Emden-Fowler type equations via uniform Haar wavelet resolution technique, \textit{J. Comput. Appl. Math.}, \textbf{376}, 2020, 112836.

\bibitem{Tomar2022} Tomar, S.; Singh, M.; Ramos, H.; Wazwaz, A.M. Development of a new iterative method and its convergence analysis for nonlinear fourth-order boundary value problems arising in beam analysis. \textit{Math. Methods Appl. Sci.} 2022, SI, 1–9.

	\bibitem{VermaAK2021} A. K. Verma, B. Pandit, R. P. Agarwal, Analysis and computation of solutions for a class of nonlinear SBVPs arising in epitaxial growth, \textit{Mathematics}, \textbf{9}, 2021, 774.

\bibitem{VermaAK2021-2} A. K. Verma, B. Pandit, R. P. Agarwal, On multiple solutions for a fourth order nonlinear singular boundaryvalue problems arising in epitaxial growth theory, \textit{Math. Meth. Appl. Sci.}, \textbf{44}, 2021, 5418-5435.

	\bibitem{Viswan2013} K.N.S.K. Viswanadham, S. Ballem, Numerical solution of fourth order boundary value problems by Galerkin method with cubic b-splines,  \textit{Int. J. Eng. Sci. innov. Technol.}, 2, 2013, 41-53. 


\bibitem{Viswana2010} K. Viswanadham, P.M. Krishna, R.S. Koneru, Numerical solutions of fourth order boundary value problems by Galerkin method with quintic b-splines, {\it Int. J. Nonlin. Sci.}, \textbf{10}, 2010, 222-230. 

 \bibitem{Wanas2020} A. K. Wanas,  Horadam polynomials for a new family of $\lambda$-pseudo bi-univalent functions associated with Sakaguchi type functions, \textit{Afrika Matematika}, 2020. \url{https://doi.org/10.1007/s13370-020-00867-1}

\bibitem{Wazwaz2010} A-M, Wazwaz,  The Numerical Solution of Special Fourth Order Boundary Value Problems by the Modified Decomposition Method, \textit{Int. J. Computer Math.},  79, 2002, 345-356.  

 \bibitem{Webb2008} Webb,J.; Infante, G.; Franco, D. Positive solutions of nonlinear fourth-order boundary-value problems with local and non-local boundary conditions. Proc. R. Soc. Edinb. Sect. A Math. 2008, 138, 427–446.



\bibitem{Xie2013} J. Xie and Z. Luo,  Solutions to a boundary value problem of a
fourth-order impulsive differential equation, \textit{Boundary Value Problems}, 2013:154.

\bibitem{Zahra2011} Zahra W.K., A smooth approximation based on exponential
spline solutions for nonlinear fourth order two point boundary
value problems.\textit{ Appl Math Comput} 217, 2011, 8447–8457.

\bibitem{Zhang2022} Y. Zhang, L. Chen, Positive solution for a class of nonlinear fourth-order boundary value problem, \textit{ AIMS Mathematics}, 8, 2022, 1014–1021.



















		













%
%
%
%
%
%
%
%
%
%
%
%
%
%
%
%
%
%
%
%
%
%
%
%
%
%
%
%
%
%
%
%
%
%
%
%
%
%
%
%
%
%
%
%
%
%
%
%
%
%
%
%
%
%
%
%
%
%
%
%
%
%
%
%
%
%
%

%
%
%
%
%
%

%
%
%
%
%
%
%
%
%
%
%
%
%
%
%
%
%
%
%
%
%
%
%
%
%
%
%
%
%
%
%
%
%
%
%
%
%
%
%
%
%
%
%
%
%
%
%
%
%
%
%
%
%
%
%
%
%
%
%
%
%
%
%
%
%


%
%
%
%
%
%
%
%
%
%
%
%
%
%
%
%
%
%
%
%
%
%
%
%
%
%
%
%
%
%
%
%
%

%
%
%
%
%
%
%
%
%
%
%
%
%
%
%
%
%
%
%
%
%
%
	\end{thebibliography}
\end{document}